\documentclass[final]{siamltex}

\usepackage{graphicx}
\usepackage{amsmath}
\usepackage{amssymb}

\newcommand{\ep}[1]{\ensuremath{\varepsilon^{#1}}}
\newcommand{\mover}[2]{\ensuremath{\left(\begin{array}{c}{#1}\\{#2}
\end{array}\right)}}

\title{An investigation of global radial basis function collocation methods applied to Helmholtz problems\thanks{The first author was supported by a grant from The Swedish 
Research Council. The second author was funded by the graduate school in 
Mathematics and Scientific Computing.}}

\author{Elisabeth Larsson\footnotemark[7] \and Ulrika Sundin\footnotemark[7]}

\begin{document}

\pagestyle{myheadings}
\thispagestyle{plain}
\markboth{ELISABETH LARSSON AND ULRIKA SUNDIN}
         {RBF APPROXIMATIONS OF THE HELMHOLTZ EQUATION}

\maketitle

\renewcommand{\thefootnote}{\fnsymbol{footnote}}
 
\footnotetext[7]{Scientific Computing, Department of Information Technology, 
Uppsala University,
Box~337, SE-751~05 Uppsala, Sweden ({\tt elisabeth.larsson@it.uu.se}, 
{\tt ulrika.sundin@it.uu.se}).}

\begin{abstract}
  Global radial basis function (RBF) collocation methods with inifinitely smooth basis functions for partial differential equations (PDEs) work in general geometries, and can have exponential convergence properties for smooth solution functions. At the same time, the linear systems that arise are dense and severely ill-conditioned for large numbers of unknowns and small values of the shape parameter that determines how flat the basis functions are. We use Helmholtz equation as an application problem for the theoretical analysis and numerical experiments. We analyse and characterise the convergence properties as a function of the number of unknowns and for different shape parameter ranges. We provide theoretical results for the flat limit of the PDE solutions and investigate when the non-symmetric collocation matrices become singular. We also provide practical strategies for choosing the method parameters and evaluate the results on Helmholtz problems in a curved waveguide geometry.
  % characterize the convergence properties as a func of N and for different shape parameter ranges
  %
\end{abstract}
\begin{keywords} 
Radial basis function, Helmholtz equation, shape parameter, flat limit, error estimate
\end{keywords}

\begin{AMS}
  % 65Nxx Partial differential equations, boundary value problems
  % 35 Spectral, collocation and related methods 
  % 41A30 Approximation by other special function classes
  % 65Dxx Numerical approximation and computational geometry
  % 15 Algorithms for functional approximation  

65N35, 65D15, 41A30
\end{AMS}

\pagenumbering{arabic}

\section{Introduction}
We started writing this paper in 2004. Some of the results can be found in the MSc thesis of the second author~\cite{Pettersson03}. At that time, the first paper on the flat radial basis function (RBF) interpolation limit~\cite{DriFo02} had just been published, and most of the work on the paper about multivariate flat RBF limits~\cite{LaFo05} was done, but the paper was not published yet. The focus of research in RBF-based methods for partial differential equations (PDEs) was on global collocation methods, and we were interested in the limit behavior for RBF approximations to PDEs. Then the manuscript ended up 'in a drawer' due to various circumstances, and we came to pick it up again 15 years later. The current research focus has shifted to localized RBF-methods such as RBF-generated finite difference methods (RBF-FD)~\cite{FoFly15} and RBF partition of unity methods (RBF-PUM)~\cite{LaShchHe17}. However, we think that the results in this paper, even though they are on global RBF methods, provide insights that are generally useful also today. The objectives of the work are 
% Currently more local but still interesting
\begin{itemize}
\item to investigate the approximation errors theoretically and numerically to gain understanding both about the flat limit, the convergence properties, and the dependence on the shape parameter,
\item to identify the gaps between theoretical results and numerical behavior,
\item to provide practically useful strategies for choosing the method parameters and assessing the results.  
\end{itemize}
The outline of the paper is as follows: In Section~\ref{sec:mod}, we define three different Helmholtz test problems that are used throughout the paper. In Section~\ref{sec:coll} we derive the systems of equations for non-symmetric and symmetric collocation. Section~\ref{sec:sing} is devoted to cases where the non-symmetric collocation matrix is singular, and in Section~\ref{sec:limit}, we discuss the limit properties. How to prove these properties is sketched in Appendix~\ref{sec:A}. Section~\ref{sec:theor} contains a combination of theoretical error estimates, and more heuristic error approximations. Then in Section~\ref{sec:exp}, we provide numerical results as well as practical strategies for method parameter selection. The paper ends with a discussion of the results in Section~\ref{sec:end}.

\section{Generic and specific model problems}\label{sec:mod}
Throughout the paper, we consider time-independent, linear, partial 
differential equations (PDEs). We assume that the PDE equation(s), together with the different boundary equations can be summarized as
\begin{equation}
\mathcal{L}^iu(\underline{x})=f^i(\underline{x}), \quad \underline{x}\in \Omega^i,\quad i=1,\ldots,N_{\mathrm{op}},
\label{eq:general}
\end{equation}   
where $\mathcal{L}^i$ is a linear operator, $u$ is the solution function, 
$f^i$ is a given function, $\underline{x}=(x_1,\ldots,x_d)\in\mathbb{R}^d$, and $\Omega^i\subseteq\bar{\Omega}$ is a region in the computational domain or a boundary segment.

To give examples and illustrate specific properties, we use a series of 
Helmholtz problems of increasing complexity. The Helmholtz equation models 
time-harmonic wave propagation, and in all cases, we consider wave guide problems with a wave originating 
from a source at the left boundary and propagating to the right. We allow 
reflected waves from the interior of the domain to propagate back to the left and 
out through the left boundary, but no waves may enter from outside the right 
boundary. The main reasons for our choice of model problems are the following:
\begin{itemize}
\item There is one problem parameter, the wavenumber $\kappa$, that can be varied
to study its relation to the RBF method parameters.  
\item A Helmholtz problem is generally more difficult to solve than a Laplace or
Poisson problem, especially for large wavenumbers, due to the indefiniteness of
the operator, the wave nature of the solution, and the typically more complicated
boundary conditions.
%\item Due to previous experience of solving Helmholtz problems using finite 
%difference methods~\cite{??}, we have access to a finite difference based reference 
%code, and we are also, in the future, interested in developing hybrid techniques.
\end{itemize}
The Helmholtz PDE is in all examples given by 
\begin{equation}
\mathcal{L}^1u(\underline{x}) = -\Delta u(\underline{x}) - \kappa^{2}u(\underline{x}) = 0, \quad \underline{x}\in \Omega^1=\Omega.
\label{eq:HZ}
\end{equation}
%
%
% Each problem. Computational domain, boundary conditions, analytical solution.
%

The first and simplest model problem is one-dimensional, with $\Omega=(0,1)$. The
non-reflecting (or radiation) boundary conditions are given by
\begin{eqnarray}
\mathcal{L}^2u(x)&=&\displaystyle -\frac{d u}{d x}(x) - i\kappa u(x)  = -2i\kappa,  \quad x=0,\label{eq:1Dbca}\\
\mathcal{L}^3u(x)&=&\displaystyle \phantom{-}\frac{d u}{d x}(x)  - i\kappa u(x)  =  \phantom{-}0,  \phantom{i\kappa}\quad x=1,
\label{eq:1Dbcb}
\end{eqnarray}
and the analytical solution is $u(x)=\exp(i\kappa x)$, if $\kappa$ is constant.

The second problem is two-dimensional with a rectangular domain 
$\Omega=(0,L_1)\times(0,1)$. At the top and bottom boundaries, we use the Dirichlet 
boundary condition 
\begin{equation}
\mathcal{L}^4u(\underline{x})=u(\underline{x})=0,\quad \underline{x}=(0,x_2) \mbox{ or } \underline{x}=(L_1,x_2),
\label{eq:L3}
\end{equation}
indicating that we consider a waveguide type of problem. The conditions at the left 
and right boundaries are
\begin{eqnarray}
%\begin{array}{lcrcll}
\mathcal{L}^2u(\underline{x})&=&\displaystyle-\frac{\partial u}{\partial x_2}(\underline{x})-i\beta_m u(\underline{x})=
-2i\beta_m\sin(\alpha_m x_1), \quad \underline{x}=(x_1,0),\\
\mathcal{L}^3u(\underline{x})&=&\phantom{-}\displaystyle\frac{\partial u}{\partial x_2}(\underline{x})-i\beta_m u(\underline{x})=\phantom{-}0,\phantom{i\beta_m\sin(\alpha_m x_1)}\quad
 \underline{x}=(x_1,1),
%\end{array}
\label{eq:2Da}
\end{eqnarray}
where $\alpha_m=\frac{m\pi}{L_1}$, $\beta_m=\sqrt{\kappa^2-\alpha_m^2}$, and $m\geq 1$ 
is an integer. These conditions allow for just one propagating mode in the solution,
which is given by $u(\underline{x})=\exp(i\beta_m x_2)\sin(\alpha_m x_1)$, assuming
a constant $\kappa$. It should be noted that if $\kappa$ and $m$ are chosen such that $\beta_m=0$, the problem is not well-defined, and we avoid such combinations in the experiments.

The third and final problem is also two-dimensional, but the domain $\Omega$ 
is now enclosed between two curves $\gamma_1(x_2) < x_1 < \gamma_2(x_2)$,
$x_2\in(0,1)$, see Figure~\ref{fig:1}. The Dirichlet condition~(\ref{eq:L3}) is modified to hold at $\gamma_1$ and $\gamma_2$.
\begin{equation}
\mathcal{L}^4u(\underline{x})=u(\underline{x})=0,\quad \underline{x}=(\gamma_j(x_2),x_2), \quad j=1,2.
\label{eq:L3b}
\end{equation}
At the left
and right boundary, we use so called Dirichlet--to--Neumann map (DtN) radiation
boundary conditions~\cite{KeGi89}
%
% Note: I have to specify at which point \psi_m is evaluated.
% The eigenmodes should be normalized with \sqrt{2}
%
\begin{equation}
\begin{array}{lclcrl}
 \mathcal{L}^2u(\underline{x})&=&\displaystyle-\frac{\partial u}{\partial x_2}
      -i\sum_{m=1}^\infty \beta_m\langle u(\cdot,0),\psi_m^{0} \rangle\psi_m^{0}(x_1)\\
      &=&-2i\displaystyle
\sum_{m=1}^\infty A_m \beta_m \psi_m^0(x_1), & x_2=0,\\               
 \mathcal{L}^3u(\underline{x})&=&\displaystyle\frac{\partial u}{\partial x_2}-i\sum_{m=1}^\infty
\beta_m\langle u(\cdot,1),\psi_m^1 \rangle\psi_m^1(x_1)=0, & x_2=1,
\end{array}
\label{eq:DtN}
\end{equation}
where, for a fixed $x_2$, the modes 
$\psi_m^{x_2}=\sqrt{2}\sin(\alpha_m (x_1-\gamma_1(x_2))$, with 
$\alpha_m=\frac{m\pi}{\gamma_2(x_2)-\gamma_1(x_2)}$. 
The inner product is given by
\begin{equation}
\langle u(\cdot,x_2),\psi_m^{x_2} \rangle=
 \int_{\gamma_1(x_2)}^{\gamma_2(x_2)} u(x_1,x_2)\psi_m^{x_2}(x_1)\,dx_1,
\end{equation}
and the amplitudes $A_m = \psi_m^0(x_s)$, where $x_s$ is the position of the 
source in the vertical coordinate. 
The amplitudes are chosen to emulate a point source. The DtN conditions 
allow for any combination of modes to move transparently through the vertical 
boundaries. For practical and
computational reasons, the infinite sum is truncated at 
$\mu_{x_2}=\lfloor\frac{\kappa (\gamma_2(x_2)-\gamma_1(x_2))}{\pi}\rfloor$. For a discussion of the assumptions behind 
this truncation and these particular DtN conditions, see~\cite{OttoLa99}.
%
% Use some exotic domains just to produce some great pictures perhaps?
%
\begin{figure}[!htb]
\centering
\includegraphics[width=0.5\textwidth]{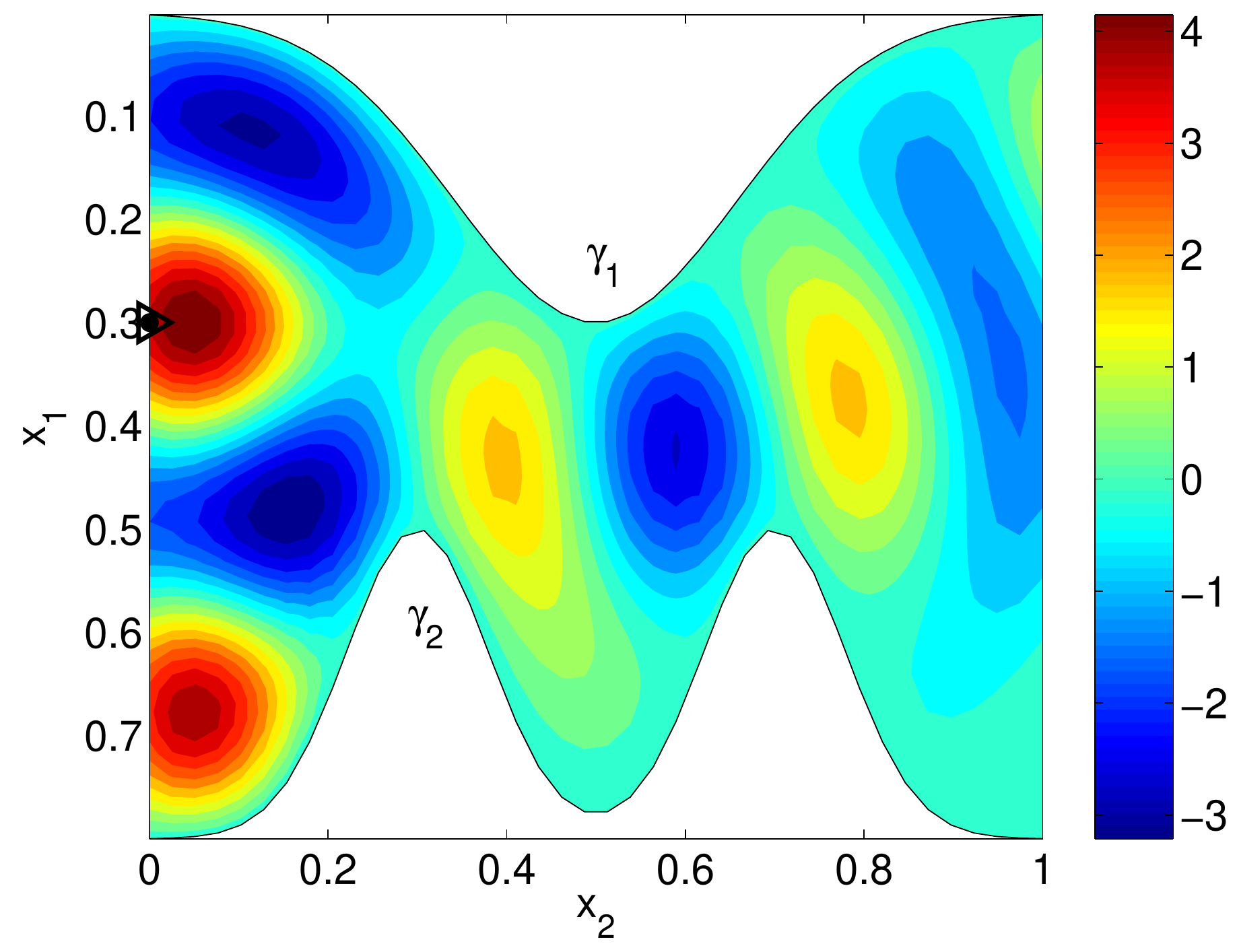}
%\includegraphics[width=0.5\textwidth]{figure1_bw.eps}
%
% Combine with evaluation grid or node points or something
%
\caption{Wave propagation in an M-shaped duct. The source position is indicated by the marker at the left boundary and the wave number is $\kappa=6\pi$.
The real part of the solution is displayed.}
\label{fig:1}
\end{figure}

\section{The RBF approximations}\label{sec:coll}
In this section, we first describe Kansa's non-sym\-metric collocation 
method~\cite{Kansa90} for our model problems. The main advantage of the 
non-symmetric
collocation method is its simplicity. This is also why we use this method
for both numerical and theoretical studies throughout this paper.
However, an argument against using non-symmetric collocation is that
the RBF approximation matrix, in rare cases~\cite{HonSch01},
can become singular. This is discussed further in Section~\ref{sec:sing}.
To avoid singularity, symmetric collocation~\cite{Wu92,Fass97,FraSch98} can 
instead be employed. This is slightly
more involved, especially with non-trivial operators, which is why we
include an example of how to do this for the one-dimensional model problem.
 
\subsection{Non-symmetric collocation}
\label{sec:non-symm}
When we use non-symmetric collocation to discretize the 
problem~(\ref{eq:general}), the RBF approximant is given by
\begin{equation}
s(\underline{x}) = \sum_{j=1}^N \lambda_j \phi(\ep{}\|\underline{x}-\underline{x}_j\|)=\sum_{j=1}^N \lambda_j \phi_j(\underline{x}),
\label{eq:RBFapprox}
\end{equation}
where $\underline{x}_j$, $j=1,\ldots,N$ are the RBF center points and $\ep{}$ is the 
shape parameter. 
%RBFs without shape parameter are discussed later, and in 
%that case, $\ep{}$ can be omitted.
The collocation conditions are imposed at the $N$ center points. 
Let $\underline{x}_j^k$, $j=1,\ldots,N_k$ be the subset of center points that 
belong to the region or section $\Omega^k$. The corresponding operator is used for 
collocation, and we get the equations 
%Each point
%$\underline{x}_i$ is associated with only one region or section $X_k$ and 
%the corresponding condition is used. For each point, we get an equation
\[\mathcal{L}^ks(\underline{x}^k_i)=\sum_{j=1}^N \lambda_j \mathcal{L}^k\phi_j(\underline{x}^k_i)=f^k(\underline{x}^k_i),\quad i=1,\ldots,N_k,\quad k=1,\ldots,N_{\mathrm{op}}.\]
%,\quad i=1,\ldots,N,\quad \underline{x}_i\in X_k.\]
If the points are ordered according to the set affiliation, we get a system
of equations, $M\underline{\lambda}=\underline{f}$, with the following general block structure
\begin{equation}
\left(\begin{array}{ccc}
&\mathcal{L}^1\phi& \\ &\vdots& \\ &\mathcal{L}^{N_{\mathrm{op}}}\phi&
\end{array}\right)
\left(\begin{array}{c}
\rule{0ex}{2.5ex}\\\underline{\lambda}\\\rule{0ex}{2.5ex}\hspace{0mm}
\end{array}\right)=
\left(\begin{array}{c}
\underline{f}^1\\ \vdots \\ \underline{f}^{N_{\mathrm{op}}}
\end{array}\right),
\label{eq:sys}
\end{equation}
where the block $\mathcal{L}^k\phi$ is of size $(N_k\times N)$.

Applying the operators in the specific model problems to the RBFs is 
straightforward, except for the DtN operators in~(\ref{eq:DtN}).
The left boundary condition applied to one of the RBFs and evaluated at the 
point $\underline{x}=(x_1,0)$ takes the form
\[\mathcal{L}^2\phi_j(\underline{x})=
-\frac{\partial \phi_j}{\partial x_2}(\underline{x})
      -i\sum_{m=1}^{\mu_0} \beta_m\langle \phi_j(\cdot,0),\psi_m^0 \rangle
\,\psi_m^0(x_1).\]
To form the whole block $\mathcal{L}^2\phi$, we need to evaluate $\mu_0\cdot N$ inner 
products. This cannot in general be done analytically for infinitely smooth RBFs such 
as multiquadrics, inverse quadratics, or Gaussians. 

One of our aims with choosing the Helmholtz model problems was to see if 
using RBFs would make it difficult to implement non-trivial boundary 
conditions. There are no fundamental issues preventing implementation of boundary conditions involving linear functionals applied to the basis functions. A practical issue is that the computational cost for the quadrature is quite large, although linear 
in $N$. In Section~\ref{sec:exp}, we investigate how accurately we need to compute the inner products to not destroy the overall accuracy of the solution.
The experiments show that we need to compute the inner products more accurately than the overall error tolerance, which increases the cost further. 
%PROBABLY WE CAN DO GAUSSIANS AGAINST A SINE ANALYTICALLY BY MAKING COMPLEX AND TAKING REAL PART. Not analytic for MQ.
%tthere are no special difficulties, otherwise than that matrix
%elements need to be computed very accurately.

%N=866 mu=4 866*4*2(boundaries)=6928
%Long time, but linear in N and not so terribly dep on ep. Good that it w`orks
%bad that it takes time.

%1e-10 99.2\% of the time 95.3 for 1e-4. 80\% total time
%
% Problem with index i and complex quantity i. Go through everything and 
% check.
%

\subsection{Symmetric collocation}
Non-singularity of the RBF approximation matrix can be ensured through
symmetric 
collocation~\cite{Wu92,Fass97,FraSch98}. The idea is to view the RBF
$\phi(\ep{}\|\underline{x}-\underline{\xi}\|)$ as a function of two variables
$\psi(\underline{x},\underline{\xi})$. Then in the ansatz for the 
RBF approximation, for each basis function, the operator
corresponding to its center location is applied to the second argument 
of the basis function. Since we consider complex operators, we also need
to conjugate the operators in order to get a Hermitian matrix in the end. 
The approximation then takes the form 
\[s(\underline{x})=\sum_{k=1}^{N_{\mathrm{op}}}\sum_{j=1}^{N_k}\lambda^k_j
\overline{\mathcal{L}^k_\xi}\psi(\underline{x},\underline{x}^k_j).\]
For the one-dimensional Helmholtz problem, 
collocation with this ansatz leads to a system of equations
with the following structure
%
% Consistency. Round ( or square [ for matrices
%
% Perhaps I need to use s(x,ep) everywhere also.
%
\[\left(\begin{array}{lll}
\mathcal{L}^1_x\overline{\mathcal{L}^1_\xi}\psi & \mathcal{L}^1_x\overline{\mathcal{L}^2_\xi}\psi & \mathcal{L}^1_x\overline{\mathcal{L}^3_\xi}\psi\\
\mathcal{L}^2_x\overline{\mathcal{L}^1_\xi}\psi & \mathcal{L}^2_x\overline{\mathcal{L}^2_\xi}\psi & \mathcal{L}^2_x\overline{\mathcal{L}^3_\xi}\psi\\
\mathcal{L}^3_x\overline{\mathcal{L}^1_\xi}\psi & \mathcal{L}^3_x\overline{\mathcal{L}^2_\xi}\psi & \mathcal{L}^3_x\overline{\mathcal{L}^3_\xi}\psi\\
\end{array}\right)
\left(\begin{array}{l}
\underline{\lambda}^1\\\underline{\lambda}^2\\\underline{\lambda}^3
\end{array}\right)=
\left(\begin{array}{c}
\underline{0}\\-2i\kappa\\0
\end{array}\right)
,\]
where the block $\mathcal{L}^j_x\overline{\mathcal{L}^k_\xi}\psi$ is of size $(N_j\times N_k)$. To 
see that the coefficient matrix $M$ really is Hermitian, we can use the following differentiation rules for the RBFs
\begin{eqnarray}
  \frac{\partial^n}{\partial \xi^n}\psi(\underline{x}_j,\underline{x}_k) &=& (-1)^n\frac{\partial^n}{\partial x^n}\psi(\underline{x}_j,\underline{x}_k),\\
  \frac{\partial^n}{\partial x^n}\psi(\underline{x}_k,\underline{x}_j) &=& (-1)^{n}\frac{\partial^n}{\partial x^n}\psi(\underline{x}_j,\underline{x}_k).
\end{eqnarray}  
We can then show for the different blocks in the matrix that the matrix elements satisfy $m_{jk}=\overline{m}_{kj}$. As an example, for elements in the first two off-diagonal blocks we get
{\small
\begin{align*}
\mathcal{L}^1_x\overline{\mathcal{L}^2_\xi}\psi(\underline{x}_j,\underline{x}_k)&=
(-\frac{\partial^2}{\partial x^2}-\kappa^2)
(-\frac{\partial}{\partial \xi}+i\bar{\kappa})\psi(\underline{x}_j,\underline{x}_k)
= (-\frac{\partial^2}{\partial x^2}-\kappa^2)
(\frac{\partial}{\partial x}+i\bar{\kappa})\psi(\underline{x}_j,\underline{x}_k),
%(-\frac{\partial^3}{\partial x^3}-\kappa^2\frac{\partial}{\partial x}
%-i\bar{\kappa}\frac{\partial^2}{\partial x^2}-i\kappa|\kappa|^2)\psi(\underline{x}_j,\underline{y}_k),\\
\\
\overline{\mathcal{L}^2_x\overline{\mathcal{L}^1_\xi}}\psi(\underline{x}_k,\underline{x}_j)&=(-\frac{\partial}{\partial x}+i\bar{\kappa})
(-\frac{\partial^2}{\partial \xi^2}-\kappa^2)\psi(\underline{x}_k,\underline{x}_j)
=(\frac{\partial}{\partial x}+i\bar{\kappa})
(-\frac{\partial^2}{\partial x^2}-\kappa^2)\psi(\underline{x}_j,\underline{x}_k).
%(\frac{\partial^3}{\partial x^3}+\bar{\kappa}^2\frac{\partial}{\partial x}
%+i\kappa\frac{\partial^2}{\partial x^2}+i\bar{\kappa}|\kappa|^2)\psi(\underline{y}_k,\underline{x}_j),
\end{align*}
}
%where we used the relation $\frac{\partial^p \psi}{\partial y^p}=(-1)^p\frac{\partial^p \psi}{\partial x^p}$. Noting that $\frac{\partial^p \psi}{\partial x^p}(\underline{x},\underline{y})=(-1)^p\frac{\partial^p \psi}{\partial x^p}(\underline{y},\underline{x})$, we find that these
%two elements are in fact complex conjugates.

Apart from the important non-singularity property, limited numerical 
experiments also show that the conditioning is slightly better (one order of 
magnitude) than for the non-symmetric method. However, the error curves, as 
functions of both $x$ and $\ep{}$, are close to identical.

It would be complicated to implement the symmetric collocation
method for the two-dimensional problem with DtN boundary conditions. 
It would also be even more costly than for the non-symmetric case, because of the increased number of integrals
to 
compute.
As mentioned for example in~\cite{HonSch01}, when 
using non-symmetric collocation, singular matrices are hardly ever observed. Due to its simplicity, non-symmetric collocation is more widely used than symmetric collocation.
In the following, we choose to study the properties of the non-symmetric 
collocation method.

\section{Singularity of the RBF collocation matrix}
\label{sec:sing}
%
% Change so that we say that it is obvious that we can achieve singularity, 
% but harder to show for elliptic, but also there it has been done.
%
As already stated, the RBF collocation matrix may become singular with the 
non-symmetric collocation approach.
%The main argument against using non-symmetric collocation is that the 
%collocation
%matrix may become singular. 
This becomes particularly clear for problems with a parameter that can be
varied freely as for our Helmholtz examples.
For the one-dimensional Helmholtz model problem,
we can in fact show that for any given node distribution 
(with distinct nodes) there are always wavenumbers $\kappa$
that lead to a singular collocation matrix.

To get the equations in an appropriate form for eigenvalue analysis, we multiply the PDE~\eqref{eq:HZ} with $-1$ and the boundary conditions~(\ref{eq:1Dbca}) and~\eqref{eq:1Dbcb} with $i\kappa$.
After collocation with the PDE at the interior points $x^1_j$,
and the boundary conditions at the boundary points, we get a collocation
matrix $M$ with elements
\[m_{jk}=\left\{\begin{array}{lcll@{,\quad}l@{\quad}l}
\kappa^{2}\phi_k(x^1_j)&+&&\phi_k^{\prime\prime}(x^1_j) & j=1,\ldots,N-2,&k=1,\ldots,N,\\
\kappa^2 \phi_k(0)&-&i\kappa \phi_k^\prime(0)&& j=N-1,&k=1,\ldots,N,\\
\kappa^2 \phi_k(1)&+&i\kappa \phi_k^\prime(1)&& j=N,&k=1,\ldots,N.
\end{array}\right.\]
We can express $M$ as a matrix polynomial in $\kappa$,
\[M=\kappa^2A+\kappa iB+C,\]
where $A$, $B$, and $C$ are real matrices. Furthermore, $A$ is the usual RBF 
interpolation matrix. The question of singularity of $M$
can be posed as a quadratic eigenproblem
\begin{equation}
(\kappa^2A+\kappa iB+C)\underline{v}=\underline{0}.
\label{eq:qep}
\end{equation}
For standard RBFs and distinct points, $A$ is non-singular. By introducing
$\underline{w}=\kappa\underline{v}$ we can then 
reformulate~(\ref{eq:qep}) as a standard eigenvalue problem
\[\left(\begin{array}{rr}0 & I\\-A^{-1}C & -iA^{-1}B\end{array}\right)
  \left(\begin{array}{c}\underline{v} \\ \underline{w}\end{array}\right)=\kappa
\left(\begin{array}{c}\underline{v} \\ \underline{w}\end{array}\right).\]
Solving this problem leads to $2N$ eigenvalues. That is, values of $\kappa$
for which the collocation matrix $M$ is singular. Two of the eigenvalues 
have to be $\kappa=0$ because of the scaling of the 
boundary conditions.
%%%%%%%%%%%%%%% HERE, not clear why.
By conjugating equation~(\ref{eq:qep}), we get
\[(\bar{\kappa}^2A-\bar{\kappa}iB+C)\underline{\bar{v}}=
((-\bar{\kappa})^2A+(-\bar{\kappa})iB+C)\underline{\bar{v}}=
\underline{0}.\]
That is, if $(\kappa,\underline{v})$ is an eigenvalue--eigenvector pair, then $(-\bar{\kappa},\bar{\underline{v}})$ also is.
Hence, all eigenvalues with $\mathrm{Re}(\kappa)\neq 0$ must come in pairs
($\kappa$, $-\bar{\kappa}$). Then, there may also be a number of eigenvalues
on the imaginary axis. %With real eigenfunctions. Exponential decay.
The $\kappa$ that are of interest in the Helmholtz problem are such that
$\mathrm{Re}(\kappa)>0$. We are then left with a maximum of $N-1$ potentially
interesting wavenumbers that lead to a singular problem.

In Figure~\ref{fig:eigs1D}, the eigenvalues that lead to a singular system are computed for different problem sizes using multiquadric and Gaussian RBFs. For multiquadrics, there are no eigenvalues in the region of interest, that is, real wavenumbers with solutions that are well resolved. The eigenvalues with the largest real part are closest to the real axis. These problems are resolved to $2\pi N/\kappa\approx$ 2 points per wavelength, which is the theoretical lowest possible resolution to use for a wave propagation problem.
%There are also some eigenvalues for wavenumbers close to zero.
The Gaussian approximation produces eigenvalues that are closer to the real axis, but also here the eigenvalues with large real part correspond to badly resolved problems. It should be noted that complex wavenumbers, typically with a significantly smaller imaginary part than real part, are used in practical applications to model damping within the media that the waves propagate through.
\begin{figure}[!htb]
\centering
\includegraphics[height=0.3\textwidth]{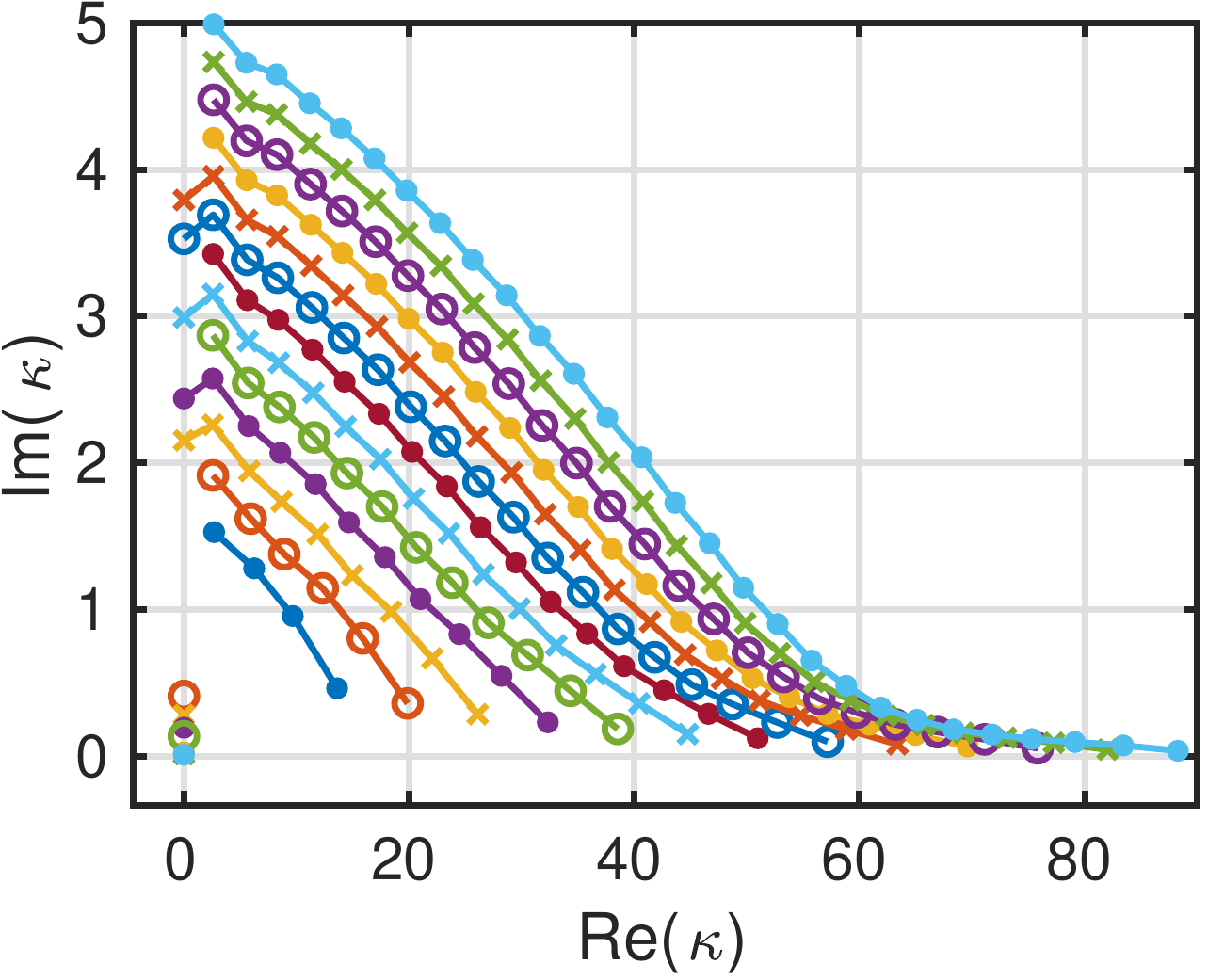}
\includegraphics[height=0.3\textwidth]{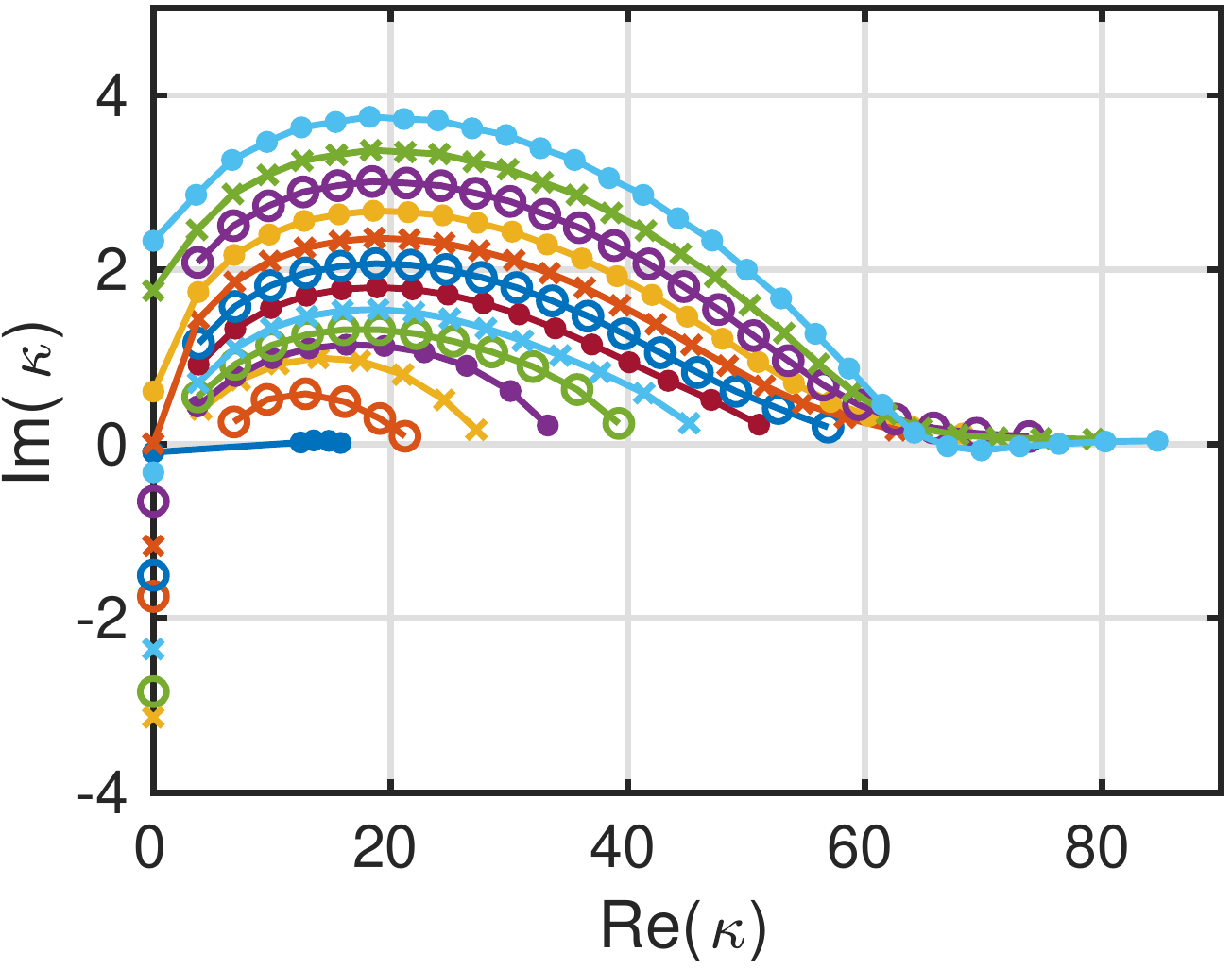}
%\includegraphics[width=0.5\textwidth]{figure1_bw.eps}
%
% Combine with evaluation grid or node points or something
%
\caption{The wavenumbers that lead to a singular system for the one-dimensional problem using $N=6,8,\ldots,30$ from bottom to top, for multiquadric RBFs with $\ep{}=5$ (left) and Gaussian RBFs with $\ep{}=10$ (right).}
\label{fig:eigs1D}
\end{figure}
% How many wavelengths for a certain kappa? 2pi is one wavelength
% The large one close to the axis is at 0.42-0.45 ppw for GS
% 0.36-0.47 ppw for MQ
% Also check the smallest eig.
% Comes closer to zero for larger N, 0.0092 in the worst case for MQ.
% 0.0054 for GS. 0.001 is min imag.

%[Can there be real eigenvalues for the 1-D problem.]
%Assume kappa real, lambda not
%(\kappa^2A+C)\lambda_r -\kappa B\lambda_i
%(\kappa^2A+C)\lambda_i +\kappa B\lambda_r
%Write the full system and see if it is non-singular.
%vec lr,li,nr,ni
%000I0x
%0000I

%Pos def for IQ, but not for MQ. Probably non-singular though.

\section{The flat RBF limit for PDE problems}\label{sec:limit}
% Here it is of interest to cite also the Finnish authors. Maybe all that we can do can be carried over to functionals of the data.

% In the previous section, singularity of the 
% RBF collocation matrix for positive values of the shape parameter was 
% discussed. When $\ep{}=0$, the basis functions are indistinguishable
% and the collocation matrix is singular. However, as shown for RBF 
% interpolation in~\cite{DriFo02} and~\cite{LaFo05}, in most cases the limit 
% solution function exists and is a finite order polynomial.

The limits of multivariate RBF interpolants as the shape parameter $\ep{}$
goes to zero were analyzed in~\cite{LaFo05,Schaback05,LeYoYo07,Schaback08,LeMiYo15}. The same type of limits for finitely smooth RBFs where also studied in~\cite{SoRiFaHi12,LeMiYo14}. It was shown that the limit 
behavior is closely related to polynomial unisolvency~\cite{Bos91} on the set 
of node points. We define
\begin{equation}
  N_{K,d}=\mover{K+d}{K},
\end{equation}
which is the dimension of the space of polynomials of degree $K$ in $\mathbb{R}^d$. If $N=N_{K,d}$, and the node set is unisolvent, then the (infinitely smooth) flat limit RBF interpolant reproduces the multivariate polynomial interpolant of degree $K$ on these nodes.

When we apply the non-symmetric collocation method to a PDE
problem, the RBF approximant has the same general form~(\ref{eq:RBFapprox}),
and we can derive corresponding results for the limit. 

In order to express the conditions for different limit results, we need to 
define two matrices, $P$ and $Q$, from which we can determine 
polynomial unisolvency and unisolvency of the discrete PDE problem. 
Let $\{p_j(\underline{x})\}_{j=1}^N$ be $N$ linearly independent monomials 
of minimal degree in
$d$ dimensions. For example, for $N=7$ and $d=2$, we can 
choose $\{1,\,x,\,y,\,x^2,\,xy,\,y^2,\,x^3\}$.
%
%Let 
%\begin{equation}
%N_{K,d}=\mover{K+d}{K}
%\end{equation}
%be the dimension of the space of polynomials of degree $K$ in $\mathbb{R}^d$.
If $N_{K-1,d}<N\leq N_{K,d}$, then the degree of $p_N(\underline{x})$ is $K$.

The set of node points $\{x_i\}_{i=1}^N$ 
satisfies polynomial unisolvency if there, for any given data at the node 
points, is a unique linear combination 
$\sum_{j=1}^N\beta_jp_j(\underline{x})$ that interpolates the data. 
This is equivalent to non-singularity of the matrix 
\begin{equation}
P=\left(
\begin{array}
[c]{cccc}%
p_{1}(\underline{x}_{1}) & p_{2}(\underline{x}_{1}) & \cdots & p_{N}%
(\underline{x}_{1})\\
p_{1}(\underline{x}_{2}) & p_{2}(\underline{x}_{2}) & \cdots & p_{N}%
(\underline{x}_{2})\\
\vdots & \vdots &  & \vdots\\
p_{1}(\underline{x}_{N}) & p_{2}(\underline{x}_{N}) & \cdots & p_{N}%
(\underline{x}_{N})
\end{array}
\right).
\end{equation}
In cases where $P$ is singular, we instead construct a minimal
non-degenerate basis~\cite{LaFo05}. Such a basis can be constructed by choosing $N$ 
monomials of smallest possible degree under the constraint that they give linearly 
independent columns in the matrix $P$. The highest selected monomial degree $M$ is then also the degree of $p_N(\underline{x})$. As an example, for $N=5$ node points all located on the line $x=y$, a minimial non-degenerate basis is  $\{1,\,x,\,x^2,\,x^3,\,x^4\}$ and $M=4$.

Unisolvency of the discrete PDE problem on the set of node points $\{x_i\}_{i=1}^N$ with respect to $\{p_j(\underline{x})\}_{j=1}^N$ requires a unique
linear combination $\sum_{j=1}^N\beta_jp_j(\underline{x})$ that satisfy the 
collocation conditions
\[\sum_{j=1}^N \beta_j \mathcal{L}^kp_j(\underline{x}^k_i)=f^k(\underline{x}^k_i),\quad i=1,\ldots,N_k,\quad k=1,\ldots,N_{\mathrm{op}}.\]
This is equivalent to non-singularity of the matrix 
\begin{equation}
Q=\left(\begin{array}{ccc}
\mathcal{L}^1p_1(\underline{x}^1_1) &\cdots &\mathcal{L}^1p_N(\underline{x}^1_1)\\
 \vdots &&\vdots\\ 
\mathcal{L}^{N_{\mathrm{op}}}p_1(\underline{x}^{N_{\mathrm{op}}}_{N_{N_{\mathrm{op}}}})&\cdots&\mathcal{L}^{N_{\mathrm{op}}}p_N(\underline{x}^{N_{\mathrm{op}}}_{N_{N_{\mathrm{op}}}})
\label{eq:Q}
\end{array}\right).
%\left[\begin{array}{c}
%\\\underline{\beta}\\\hspace{0mm}
%\end{array}\right]=
%\left[\begin{array}{c}
%\underline{f}^0\\ \vdots \\ \underline{f}^{N_c}
%\end{array}\right],
\end{equation}

As in~\cite{LaFo05}, we need the RBFs to fulfill three conditions in order to get the
results in the theorem given below. We repeat the conditions and discuss their validity
briefly here, but for a full explanation, we refer the reader to~\cite{LaFo05}. 
\begin{itemize}
\item[(I)] The RBF $\phi(r)$ can be Taylor expanded as $\phi(r)=\sum_{j=0}^\infty a_jr^{2j}$.
\item[(II)] The PDE collocation matrix $M$ in system~(\ref{eq:sys}) is non-singular in the
            interval $0<\ep{}\leq R$, for some $R>0$.
\item[(III)] Certain matrices $A_{p,J}$, built from the coefficents $a_j$ in the Taylor 
expansion of $\phi(r)$, are non-singular for $0\leq p\leq d$ and $0\leq J \leq K$.
\end{itemize}
Condition (I) is true for all infinitely smooth RBFs that are commonly used.  Condition (II) is likely to hold for some value of $R$, but the
previous section shows that $M$ can become singular at any $\ep{}$, given a specific
combination of PDE problem and node points.  
Condition (III) was shown to hold for these RBFs in~\cite{LeYoYo07}.

%Also in the PDE case, the limit of the RBF approximant is polynomial if it 
%exists. However, instead of 

%In the PDE case, the existence of a limit has a close connection to 
%unisolvency of the discrete PDE problem with respect to a polynomial basis,
%whereas the degree 

%$V_N$

%The existence of a limit is in the PDE case related to unisolvency of
%the discrete PDE problem with respect to $P_N$. 

%The degree of the limit is governed by unisolvency, which is equivalent to non-singularity of the matrix $P$,

%Because the theory is complicated to write down, we refer the reader to~\cite{} for details and give the corresponding Theorems and the differences in the proofs in the Appendix.
% 
% Why is this interesting? Because we want to be close to the limit.
%

%For the theorems we need a system matrix $L$, $A_{p,K}$, but these involve a lot. Perhaps we can just reference them and say that the condition is 
%usually fulfilled? The Taylor expansion if we are going to use that in the 
%examples. 

%
% Perhaps we need to try J0 in 2D to see if something nice happens.
%

%Give the set of conditions we need.
%
% A note that Gaussians are not special. Give an example for every case.
%
%P has a nullspace of dim and H has a nullspace of dim.

The following theorem gives the different possibilities for the limiting RBF approximant
as the shape parameter $\ep{}\rightarrow 0$.
\begin{theorem}
Assume that the RBF $\phi(r)$ fulfills conditions (I)--(III) and that the 
number of node points satisfy $N_{K-1,d}<N\leq N_{K,d}$. The degree of a minimal 
non-degenerate basis for the point set is either $K$ for a unisolvent set or $M$ for a 
non-unisolvent set. If 
\begin{itemize}
\item[(i)] $P$ and $Q$ are non-singular, the limit exists and is a polynomial of deg $K$.
If $N=N_{K,d}$ it is the unique degree $K$ polynomial solution to the discrete PDE 
problem, otherwise the final polynomial depends on the choice of RBF.
\item[(ii)] $P$ is singular, but $Q$ is non-singular, the limit exists and is an
RBF-dependent polynomial of degree $M$.

\item[(iii)]$P$ is non-singular, but $Q$ is singular, divergence will occur unless 
the right hand side $\underline{f}$ of system~(\ref{eq:sys}) happens to be in the range 
of $Q$. If there is just a single null-space polynomial $n(\underline{x})$ of degree $K$, 
the divergent term is proportional to $\ep{-2}n(\underline{x})$.

\item[(iv)] $P$ has a nullspace of dimension $m>0$ and $Q$ has a nullspace of dimension 
$p>0$, then if $m\geq p$ the limit is likely, but not certain to 
exist. If it exists it is of degree $M$. If $m<p$ divergence is likely, but not certain. 
\end{itemize}
\label{th:1}
\end{theorem}

The proof builds on the results for RBF interpolation in~\cite{LaFo05}. The 
key arguments and differences are pointed out in Appendix~\ref{sec:A}.

The uncomplicated case (i) is of course the most common and the other types are deviations
stemming from degenerate node point configurations or specific combinations of PDE problem
parameters and node points that lead to degeneracy. Below, we give an example of each
type of degeneracy for the two-dimensional Helmholtz problem given by~(\ref{eq:HZ}) and  
(\ref{eq:L3})--(\ref{eq:2Da}) with $m=1$.

%
% drawpoints.m
%
\noindent
\subsection*{Example (ii): The node set is not polynomially unisolvent} \hspace{1mm}\\[1pt]
\parbox{0.2\textwidth}{\includegraphics[width=0.2\textwidth]{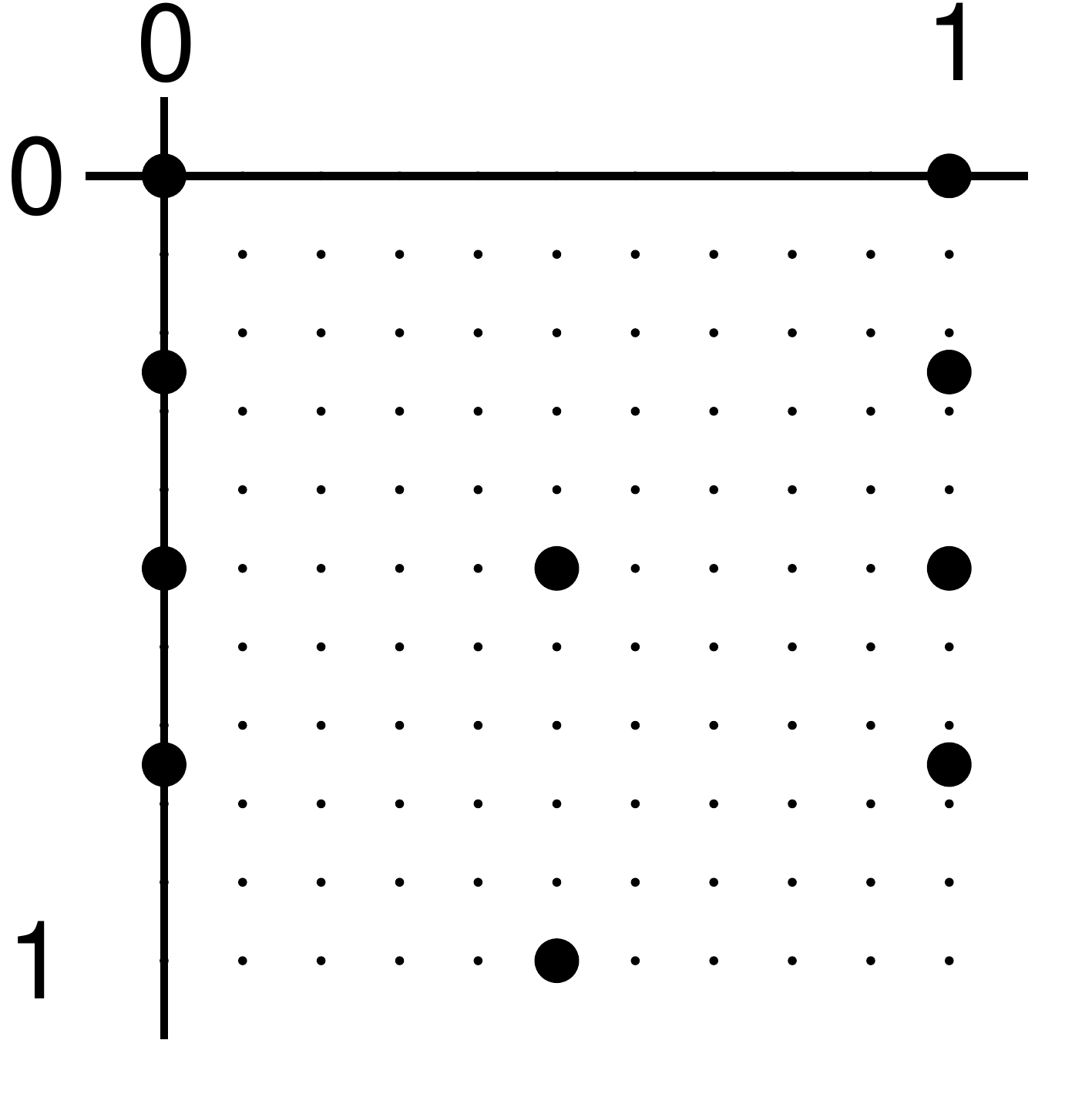}}
\parbox{0.8\textwidth}{
The $N=10$ points are $(1/2,1/2)$, $(1,1/2)$, and $(k/4,0)$,  
$(k/4,1)$, $k=0,\ldots,3$. The matrix $P$ has a nullspace defined by
$n(\underline{x})=x_2(x_2-\frac{1}{2})(x_2-1)$. A non-degenerate basis is given by $\{1,\,x_1,\,x_2,\,x_1^2,\,x_1x_2,\,x_2^2,\,x_1^3,\,x_1^2x_2,\,x_1x_2^2,\,x_1^4\}$ with $M=4$. The limit is hence a polynomial of degree four
whose coefficients depend on the choice of RBF. To illustrate what this
dependence can look like, we give the general form of the limit polynomial
$p(\underline{x})$.
}
%
% This needs to be checked numerically with the actual solution in the file.
% Kollat och det verkar stämma. b8 och b9 omkastade, men spelar ingen roll.
%
\begin{eqnarray*}
p(\underline{x})&=& %p_2(\underline{x})+
\beta_0+\beta_1x_2+\beta_2x_1+\beta_3x_2^2+\beta_4x_2x_1+\beta_5x_1^2\\
&+&\beta_6\left[-12a_3(x_2^3+3x_2x_1^2)+\frac{8a_2^2}{a_1}(x_2^3+x_2x_1^2) \right]\\
&+&\beta_7\left[-12a_3(x_1^3+3x_2^2x_1)+\frac{8a_2^2}{a_1}(x_1^3+x_2^2x_1)\right]\\
&+&\beta_8\left[-4a_3(5x_1^3+3x_2^2x_1)+\frac{8a_2^2}{a_1}(x_1^3+x_2^2x_1)\right]\\
&+&\beta_{9}\left[\phantom{\frac{1}{2}}-4a_4(9x_2^3+36x_2^2x_1+45x_2x_1^2+20x_1^3-24x_2^3x_1-40x_2x_1^3)\right.\\
&+&\left.\frac{6a_2a_3}{a_1} (3x_2^3+4x_2^2x_1+3x_2x_1^2+4x_1^3)
-\frac{72a_3^2}{a_2}(x_2^3x_1+x_2x_1^3)\right],
\end{eqnarray*}
where $a_j$ are the Taylor expansion coefficients of the RBF, and $\beta_j$ are the ten unknown coefficients that are determined by the ten discrete
PDE collocation conditions. The result is $\kappa$-dependent as 
well as RBF-dependent.

\subsection*{Example (iii): The node set is not PDE-unisolvent} \hspace{1mm}\\
% 
% From ``Ett till fall'' in mail/ulrikap
%
\parbox{0.2\textwidth}{\includegraphics[width=0.2\textwidth]{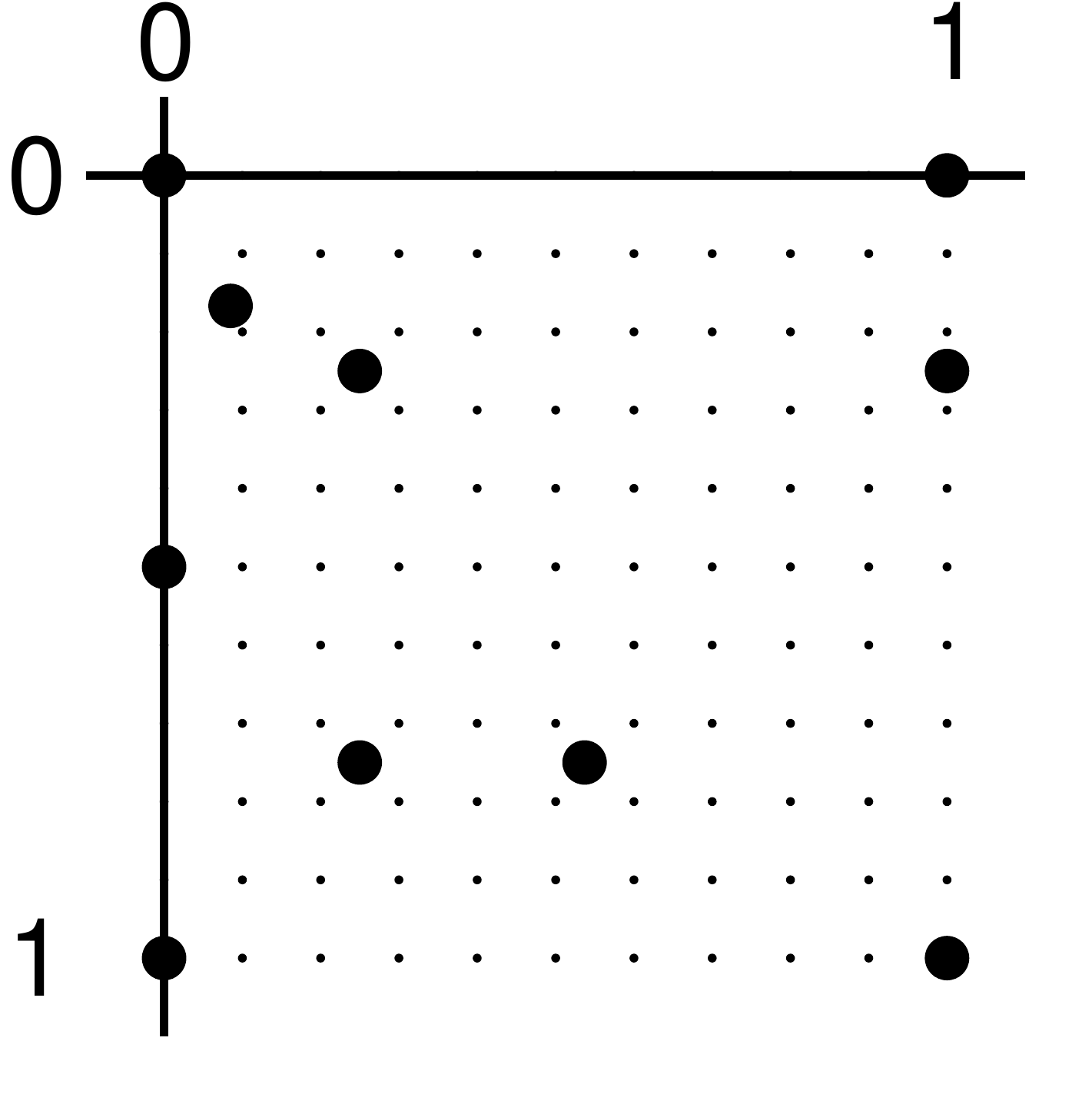}}
\parbox{0.8\textwidth}{The $N=10$ points are $(0,0)$, $(1/2,0)$, $(1,0)$, 
$(0,1)$, $(1/4,1)$, $(1,1)$, $(1/6,(2545-23\sqrt{9233})/3936)$, $(1/4,1/4)$, $(3/4,1/4)$, and $(3/4,969/1804)$. For $\kappa=4\sqrt{246}/9$ the matrix $Q$ 
has a nullspace defined by $q(\underline{x})=-\frac{5}{32}x_2(x_2+1)
+\frac{x_1}{16}(8-24x_1+3x_2+16x_1^2+4x_1x_2-7x_2^2)$.
% Kappa = 2*pi means one wave length. ppw = N/kappa/(2pi)
}
In this case, we get divergence of order $\ep{-2}$ as $\ep{}\rightarrow 0$
for all RBFs that obey conditions (I)--(III).
This can be observed not only in exact arithmetic, but also in for example a 
double precision numerical simulation. However, if we move just one of the 
points or change $\kappa$ slightly, there is no longer a nullspace.
This kind of degeneracy is very rare, since it requires very special 
combinations of wavenumber and node points.

\subsection*{Example (iv): Both $P$ and $Q$ have a nullspace} \hspace{1mm}\\
\parbox{0.2\textwidth}{\includegraphics[width=0.2\textwidth]{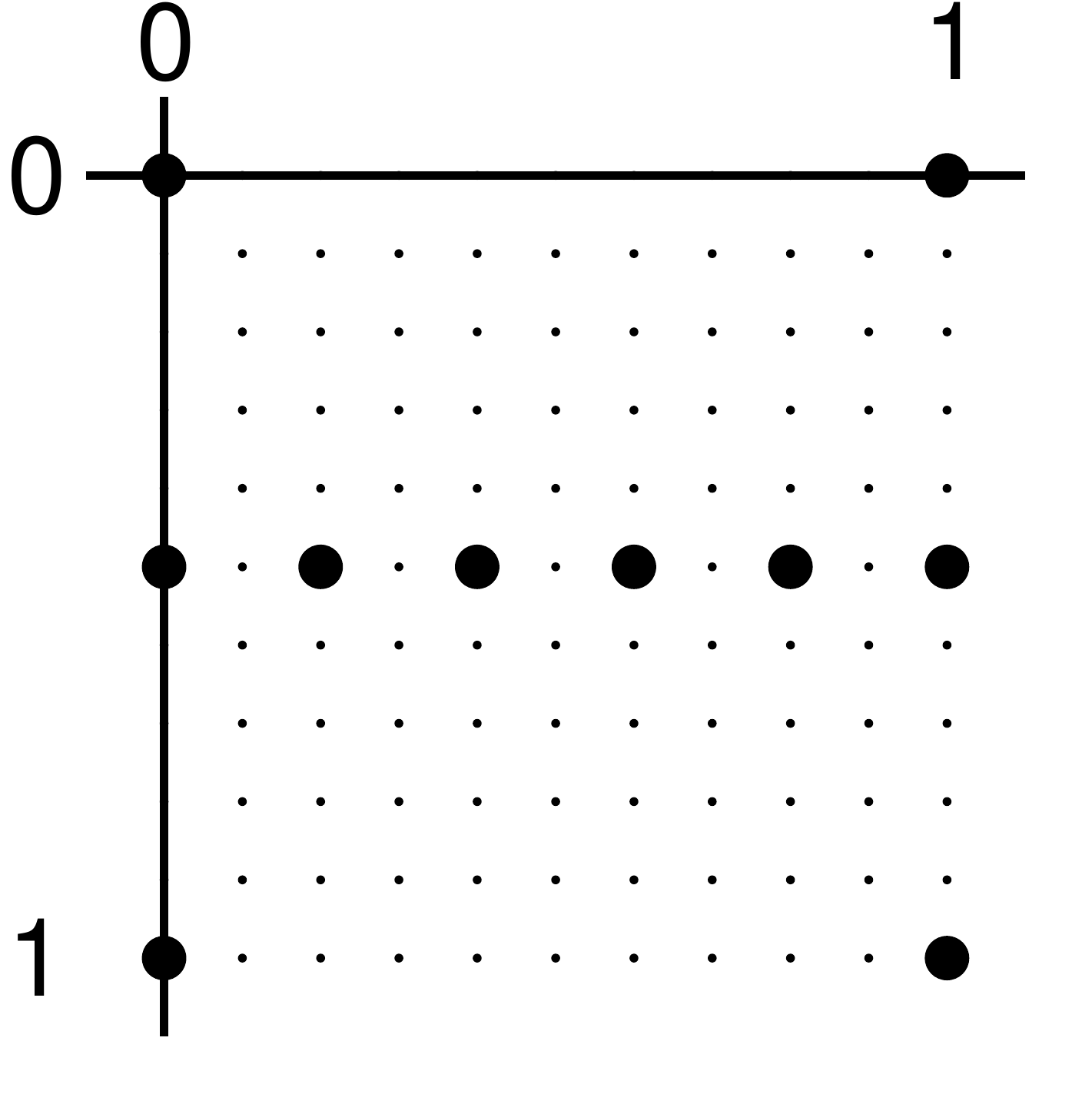}}
\parbox{0.8\textwidth}{The $N=10$ points are $(0,0)$, $(0,1)$, $(1,0)$, 
$(1,1)$ and $(1/2,k/5)$, $k=0,\ldots, 5$. The matrices $P$ and $Q$ have a
common nullspace of dimension two defined by 
$q_1(\underline{x})=x_1(x_1-\frac{1}{2})(x_1-1)$ and
$q_2(\underline{x})=x_2(x_1-\frac{1}{2})(x_2-1)$. The limit exists and 
is an RBF- and $\kappa$-dependent polynomial of degree $M=5$.

}

A practical implication of this result is that if we use node sets that are not unisolvent, e.g, Cartesian nodes, both PDE approximation and interpolation are expected to behave poorly in the small shape parameter regime. The condition number of the linear system is larger than in the unisolvent case, and the result contains a term that diverges as $\ep{}\rightarrow 0$.

An important property of interpolation with Gaussian RBFs is that it never diverges~\cite{Schaback05}. In the PDE case, this property also holds as long as $Q$ is non-singular. However, it can still be difficult to compute the limit numerically using a stable evaluation method. The RBF-QR method derived in~\cite{FoPi07,FoLaFly11}, and further explored for solving PDEs~\cite{LLHF13} uses pivoting to handle non-unisolvent cases. This means that a limit can always be computed, but it may be different from the Gaussian limit. The RBF-GA method~\cite{FoLePo13} always reproduces the correct Gaussian limit, but instead cannot handle large values of $N$. 

%
% How much of the proofs do we include? Not too much. It should all be put in
% an appendix.
%
\section{Convergence properties and error estimates}\label{sec:theor}
In this section, we investigate errors and convergence properties from different perspectives, as well as quantify how the choice of shape parameter affects the results. We start by formulating general residual-based error estimates in the following subsection. 
%we focus on convergence as a function of the number of nodes $N$, and we assume that the shape parameter is fixed during refinement. The effect of the shape parameter is studied in the following section. 

\subsection{General error estimates using Green's functions}
We define the error function as the difference between the RBF approximant and
the exact solution to the PDE problem~(\ref{eq:general})
%
% What notation did we use before? Is s a function of epsilon also.
% In the figure, I use that now.
%
\begin{equation}
e(\underline{x})=s(\underline{x})-u(\underline{x}).
\end{equation}
In the interpolation case, the error and the residual are the same, and if the
function $u(\underline{x})$ is known at a point, the corresponding error can 
be explicitly computed. In the PDE case, we can compute the residual 
for each operator, while the error is governed by the same type of PDE as the 
solution
\begin{equation}
\mathcal{L}^ie(\underline{x})=\mathcal{L}^is(\underline{x})-f^i(\underline{x})\equiv r^i(\underline{x}), \quad \underline{x}\in \Omega^i,\quad i=1,\ldots,N_{\mathrm{op}},
\label{eq:generr}
\end{equation} 
where $r^i$ are residuals. One way to find the error is to solve the above PDE
problem. However, because the residuals are zero at the collocation points, 
they are highly oscillatory and more node points are required 
to approximate the error than to solve the original PDE.  

%%%%%%%%%%% CONTINUE HERE WITH THE A POSTERIORI ESTIMATES TO USE IN THE NEXT SECTION

%As already stated in Section~\ref{sec:err}, the error in an approximate 
%solution $s(\underline{x})$ can be computed by solving the 
%PDE~(\ref{eq:generr}). However, in a realistic situation, the original 
%problem is solved at the limit of available computer resources and the 
%extra resolution needed for the error equation makes it impossible to solve.
Instead of solving the error equation, we can formulate \textit{a posteriori} error estimates in terms of the residual. Writing out the error PDE for the one-dimensional problem we get

%Can we, given a solution, decide how large the error is? Can we do it if we 
%know the residual? Does it depend on the choice of norm? We need to do some 
%more work on this.

%Picture with error in max,l2,l1 how do we compute the norms? Especially in 2D,this might prove interesting.

\begin{equation}
  \left\{\begin{array}{rcl}
-\Delta e(x) - \kappa^{2}e(x) &=&  r(x),\\
-e'(0) - i\kappa e(0)  &=& 0, \\
e'(1) - i\kappa e(1)  &=&  0.
  \end{array}\right.
\end{equation}
{A Green's function} satisfying the boundary conditions is given by
\begin{equation}
  G(x,\xi) = \frac{i}{2\kappa}e^{i\kappa|x-\xi|},
\end{equation}
with
\begin{equation}
  \frac{\partial G}{\partial \xi} = \left\{
  \begin{array}{rl}
     \frac{1}{2}e^{i\kappa|x-\xi|}, & x\geq\xi,\\
    -\frac{1}{2}e^{i\kappa|x-\xi|}, & x<\xi,
  \end{array}
  \right.
  \mbox{ and }
  \Delta_\xi G = -\frac{i\kappa}{2}e^{i\kappa|x-\xi|}-\delta(x),
\end{equation}
%\[\frac{\partial G}{\partial \xi} = \frac{i}{2\kappa}e^{i\kappa|x-\xi|},\]
%\[G(x,\xi) = \frac{i}{2\kappa}e^{i\kappa|x-\xi|},\]
%
% We will perhaps make a picture becasue it is useful.
%
such that $-\Delta_\xi G-\kappa^2G=\delta(x)$, which allows us express the error as
\begin{equation}
  e(\xi)=\int_0^1G(x,\xi)r(x)\,dx.
\end{equation}  
We can use this form of the error to formulate an error estimate as
\begin{equation}
  \|e\|_\infty \leq \int_0^1|G(x,\xi)||r(x)|\,dx=
  \frac{1}{2\kappa}\int_0^1|r(x)|dx\leq \frac{1}{2\kappa}\|r\|_\infty.
  \label{eq:errest1}
\end{equation}
%{A cheaply computable error estimate:}
%\[e(\xi) \approx G(0,\xi)\int_0^h r(x)dx+G(1,\xi)\int_{1-h}^1 r(x)dx,\]
%
%\[\max|e(\xi)|\approx \frac{1}{2\kappa(0)}\left|\int_0^hrdx\right|+
%\frac{1}{2\kappa(1)}\left|\int_{1-h}^1rdx\right|.\]
%
For the two-dimensional Helmholtz problem in a rectangular domain, the corresponding Green's function satisfying the boundary conditions is given by
\begin{equation}
  G(\underline{x},\underline{\xi})=\sum_{m=1}^\infty\frac{i}{2\beta_m}e^{i\beta_m|x_2-\xi_2|}
  \psi_m(x_1)\psi_m(\xi_1),
\end{equation}
with $-\Delta_\xi G -\kappa^2G= \delta(x_2)\sum_{m=1}^\infty\psi_m(x_1)\psi_m(\xi_1)$. Similarly as for the one-dimensional problem we define the error as
\begin{equation}
  e(\underline{\xi})=\int_0^1\int_0^{L_1}G(\underline{x},\underline{\xi})r(\underline{x})\,dx_1\,dx_2.
\end{equation}
To see how this works, we note that the vertical eigenfunctions form an orthonormal basis, and we can express the residual as
\begin{equation}
  r(x_1,x_2)=\sum_{m=1}^\infty\langle r(\cdot,x_2),\psi_m^{x_2}\rangle\psi_m^{x_2} \equiv \sum_{m=1}^\infty r_m(x_2)\psi_m(x_1).
\end{equation}
%First, we validate that the error solves the Helmholtz problem using this form of the residual
%\begin{eqnarray}
%  -\Delta_\xi e-\kappa^2e&=&\int_0^1\delta(x_2)\int_0^{L_1}\sum_{m=1}^\infty \psi_m(x_1)\psi_m(\xi_1) \sum_{n=1}^\infty r_n(x_2)\psi_n(x_1)\, dx_1\,dx_2  \nonumber\\
%  &=& \sum_{m=1}^\infty r_m(x_2)\psi_m(\xi_1)
%\end{eqnarray}
%We can also use the residual expressed in the eigenbasis
This allows us to simplify the error expression
\begin{eqnarray}
  e(\underline{\xi}) &=& \sum_{m=1}^\infty\frac{i}{2\beta_m}\int_0^1e^{i\beta_m|x_2-\xi_2|}\int_0^{L_1}\psi_m(x_1)\psi_m(\xi_1)\sum_{n=1}^\infty r_n(x_2)\psi_n(x_1)\,dx_1\,dx_2\nonumber\\
  &=&\sum_{m=1}^\infty\frac{i}{2\beta_m}\psi_m(\xi_1)\int_0^1e^{i\beta_m|x_2-\xi_2|}r_m(x_2)\,dx_2
\end{eqnarray}  
To convert this into error estimate, we first note that for $m\leq\mu_0=\lfloor\kappa L_1 /\pi\rfloor$, the horizontal wavenumber $\beta_m$ is real, and $|e^{i\beta_m|x_2-\xi_2|}|=1$, while for $m>\mu_0$, $\beta_m$ is purely imaginary and $|\int_0^1e^{i\beta_m|x_2-\xi_2|}\,dx_2|\leq |\int_0^1e^{i\beta_m|x_2-\frac{1}{2}|}\,dx_2|=\frac{2}{|\beta_m|}(1-e^{-\frac{1}{2}|\beta_m|})$. Then we get
\begin{eqnarray}
  \|e\|_\infty&\leq& \sum_{m=1}^{\mu_0}\frac{1}{\sqrt{2}\beta_m}\int_0^1|r_m(x_2)|\,dx_2 + \sum_{m=\mu_0+1}^\infty\frac{\sqrt{2}}{|\beta_m|^2}(1-e^{-\frac{1}{2}|\beta_m|})\int_0^1|r_m(x_2)|\,dx_2\nonumber\\ 
  & \leq & \sum_{m=1}^{\mu_0}\frac{1}{\sqrt{2}\beta_m}\|r_m\|_\infty+\sum_{m=\mu_0+1}^\infty\frac{\sqrt{2}}{|\beta_m|^2}(1-e^{-\frac{1}{2}|\beta_m|})\|r_m\|_\infty.
  \label{eq:err2}
\end{eqnarray}

For the two-dimensional problem in a domain with curved boundaries, we cannot provide an explicit Green's function. If we think about the curved domain as a sequence of thin almost rectangular domains, we can modify the previous estimate to get a heuristic approximation of the error
\begin{eqnarray}
  \|e\|_\infty&\approx& \sum_{m=1}^\infty\int_{\Re e(\beta_m)>0}\frac{|r_m(x_2)|}{\sqrt{2}\beta_m}\,dx_2 \nonumber\\
  &+&
  \sum_{m=1}^\infty\int_{\Im m(\beta_m)>0}\frac{\sqrt{2}}{|\beta_m|^2}(1-e^{-\frac{1}{2}|\beta_m|})|r_m(x_2)|\,dx_2.
  \label{eq:errest2}
\end{eqnarray}
We evaluate this error approximation numerically in Section~\ref{sec:exp} and find that we get surprisingly good results.

\subsection{Convergence properties for small $\ep{}$}\label{sec:smallep}
As discussed in the previous section, we approach the polynomial limit,
$s(\underline{x})=p(\underline{x})$,
as $\ep{}\rightarrow 0$. 
For polynomial \textit{interpolation} in one dimension,
the interpolation error $e_I(x)$ takes the form
\[e_I(x)=u(x)-p(x)=\frac{\prod_{j=1}^N (x-x_j)}{N!}u^{(N)}(\xi),\]
where $\xi \in (x_1,x_N)$. For equispaced points, $x_{j+1}-x_j=h$, this can 
be estimated by
\[|e_I(x)|\leq\frac{h^N}{N}\max_{\xi\in (x_1,x_N)}|u^{(N)}(\xi)|,\]
see~\cite[pp.\ 39--40]{GoOr92}. In the PDE case, the residual plays the role
of the error. By following the steps for the proof of the polynomial 
error~\cite[pp.\ 43--44]{GoOr92}, we can get a similar estimate for the
residual. 

\begin{theorem}
For a one-dimensional linear PDE problem 
\begin{equation*}
\left\{\begin{array}{rcll}
\mathcal{L}^1u(x) &=& f^1(x), & x_1 < x < x_N,\\
\mathcal{L}^2u(x) &=& f^2(x), & x=x_1,\\
\mathcal{L}^3u(x) &=& f^3(x), & x=x_N,\\
\end{array}\right.
\end{equation*}
with a polynomial solution $p(x)$ determined through collocation at the nodes $x_i$, $i=1,\ldots,N$
%[$N$ points and boundary is first and last.]
%with the limiting polynomial RBF 
%approximant $p(x)$,
the residual $r(x)=\mathcal{L}^1p(x)-f(x)$ has the form
\[r(x)=\frac{\prod_{j=2}^{N-1} (x-x_j)}{(N-2)!}r^{(N-2)}(\xi),\]
where $\xi \in (x_1,x_N)$. For equispaced points, $x_{j+1}-x_j=h$, this can 
be estimated by
\[|r(x)|\leq\frac{h^{N-2}}{N-2}\max_{\xi\in (x_1,x_N)}|r^{(N-2)}(\xi)|.\]
\end{theorem}

\begin{proof}
Let $\Psi(s)=r(s)-\frac{r(x)}{\chi(x)}\chi(s)$, where
$\chi(x)=\prod_{j=2}^{N-2}(x-x_j)$. Then $\Psi(x)=0$ and $\Psi(x_j)=0$,
$j=2,\ldots,N-2$, since $r(x_j)=0$ at all interior node points where
the equation is enforced. This means that $\Psi(s)$ has at least $N-1$
zeros. By repeated application of Rolle's theorem, we find that
$\Psi^{(N-2)}(s)$ has at least one zero. That is,
\[0=\Psi^{(N-2)}(\xi)=r^{(N-2)}(\xi)-\frac{r(x)}{\chi(x)}(N-2)!.\]
Rearranging gives the expression for $r(x)$.
\end{proof}

% emax \leq 1/2/kappa rmax
% Now we can go back to the previous equations for the error
To see how this can help us in understanding the behavior of the error for small $\ep{}$, we insert the residual estimate in the error estimate~\eqref{eq:errest1} for the one-dimensional Helmholtz problem to get
\begin{equation}
  \|e\|_\infty\leq \frac{1}{2\kappa}\frac{h^{N-2}}{N-2}\|r^{(N-2)}\|_\infty.
\end{equation}
%First note that the residual 
%is the forcing function of the error equation. Hence, the maximum error is 
%proportional to the size of the residual. However, the proportionality 
%constant is unknown and depends on $\kappa$ and also on the shape of the 
%residual. We assume the following relation for a fixed problem  
%\[\max_x|e(x)|\approx C_{\kappa}\max_x|r(x)|\leq C_{\kappa}\frac{h^{N-2}}{N-2}\max_{\xi\in (x_1,x_N)}|r^{(N-2)}(\xi)|.\]
In the flat limit, the residual is 
$r(x)=-p^{\prime\prime}(x)-\kappa^2p(x)$, where $p(x)$ is the limit 
polynomial of degree $N-1$. Then $r^{(N-2)}(x)=-\kappa^2p^{(N-2)}(x)$.  
We know that $p(x)\approx u(x)=\exp(i\kappa x)$, but
even if $p(x)$ is a very good approximation of $u(x)$, $p^{(N-2)}(x)$ 
(which is a line) is a rather poor approximation of $u^{(N-2)}(x)$.
However, numerical tests indicate that the order of magnitude is right.
That is, 
$|p^{(N-2)}|\approx |\frac{d^{N-2}\exp(i\kappa x)}{dx^{N-2}}|=\kappa^{N-2}$.
We cannot use this as a bound, but we get an approximate expression for the error in the limit
\begin{equation}
  \|e\|_\infty\approx \frac{1}{2\kappa}\frac{h^{N-2}}{N-2}\kappa^2\kappa^{N-2}
  =\frac{\kappa(\kappa h)^{N-2}}{2(N-2)}\approx \frac{1}{2}(\kappa h)^{N-1}.
\end{equation}
Note that the quantity $\kappa h$ is small only if the problem is 
adequately resolved.

For the two-dimensional problem, the limit polynomial has degree $K$ if $N_{K-1,d}<N\leq N_{K,d}$ and it is zero at the interior node points.
%, where $N_{K_0-1,d}<N_0\leq N_{K_0,d}$.
To get an estimate for the residual in terms of its derivatives, we could potentially use a sampling inequality such as~\cite[Theorem 3.5]{Madych06}, which says that for all $h\leq h_0$
\begin{equation}
  \|r\|_\infty \leq C_kh^{k}\sum_{|\sigma|=k}\|D^\sigma r\|_\infty,
  \label{eq:sampl}
\end{equation}
where $h_0$ depends on the geometry of $\Omega$. For the unit square, which we are using here, $h_0=\frac{1}{2kc_2}$ with $c_2=12$. In the discretizations that we use $h=1/(\sqrt{N}-1)$. Requiring $h\leq h_0$ leads to the following condition $k\leq (\sqrt{N}-1)/24$. We want to use the theorem for $k=K-1$, where inverting the expression for $N_{K,d}$ yields that $K=\lfloor\sqrt{2}\sqrt{N+1/8}-1.5\rfloor$. That is, the theorem does not hold in this case. 
% h=1/(\sqrt{N}-1)<=\frac{1}{2kc_2}   k <= sqrt(N)-1/24
From practical experience, the result holds also for larger $k$ (larger $h$), and we will therefore use it to approximate the residual.

%The problem is that for the unit square, $h_0=\frac{1}{2kc_2}$ with $c_2=12$, and we want to use the theorem for $k=K-1$ with $h = 1/(\sqrt{N}-1)>1/24(K-1)=h_0$, for $K=\lfloor\sqrt{2}\sqrt{N+1/8}-1.5\rfloor$. From practical experience, the result holds for much larger $h$, and we will use to approximate the residual and error.

In this case, using that $r(\underline{x})=-\Delta p(\underline{x})-\kappa^2p(\underline{x})$, and, for $|\sigma|=K-1$, $D^\sigma r(\underline{x})=-\kappa^2D^\sigma p(\underline{x})\approx -\kappa^2D^\sigma u(\underline{x})$,  we get
\[\|r\|_\infty\leq C_{K-1}h^{K-1}\kappa^2\sum_{|\sigma|=K-1}
%\|D^\gamma p\|_\infty\approx Ch^{K-1}\kappa^2\sum_{|\gamma|=K-1}
\beta_1^{\sigma_2}\alpha_1^{\sigma_1}\leq C_{K-1}\kappa^2K(\kappa h)^{K-1},\]
%
%[We can use max(alpha_1,beta_1)^(K-1)]
Combining the approximate expression for the residual with the error estimate~\eqref{eq:err2} restricted to the first mode (scaled by $1/\sqrt{2}$) gives
%\[\|e\|_\infty\approx\frac{1}{\sqrt{2}|\beta_1|}C_{K-1}\kappa^2K(\kappa h)^{K-1}\]
\[\|e\|_\infty\approx\frac{1}{2|\beta_1|}C_{K-1}\kappa^2K(\kappa h)^{K-1}\]
Numerical experiments show that $KC_{K-1}=C/(K-1)$, provides the appropriate behavior with respect to $N$ (both $K$ and $h$ are coupled with $N$). This leads to

\[\|e\|_\infty\approx\frac{C\kappa^2(\kappa h)^{K-1}}{{2}|\beta_1|(K-1)}\approx \tilde{C}(\kappa h)^K,\]
where the final expression is just to show that the dimension is similar to that of the one-dimensional error approximation.

%Then combine with e. We get 1/2beta. Then we need to determine $C_k$. We can run the estimate for different kappa and different K to get statistics. Maybe make acontour-plot to find it.
       
% We are using L_2 to design the vertical functions, but L_infty for the errors
% The derivative has a product of \beta and \alpha \leq \beta_m^\gamma
% The degree of the limit polynomial
% The residual is zero in N_0 points
% Sampling inequality gives
% Then talk about the derivatives
% Then the relation with beta.
 
% TALK ABOUT BEING CLOSE TO THE LIMIT
% THERE SHOULD BE ONE EXPERIMENT FOR 1-D and a similar for 2-D
% Mention Gaussians, the RBF-QR method
% The error approximation works as an explicit bound
% Do the same experiment in 2-D for the simplified problem. Can we guess?
% Use a sampling inequality.
Figure~\ref{fig:smallep} shows the computed errors of the one-dimensional and two-dimensional problems for small values of the shape parameter. The error behavior agrees well with the derived error approximations. For the two-dimensional problem, we also show that the error expression can be multiplied by a constant to get a very good fit to the actual error. This means that we can use $\|e\|_\infty\approx C(\kappa h)^K$ \textit{a priori} with $C=1$ to determine the necessary resolution for a given tolerance. Given at least two numerical solutions, we can also estimate the constant $C$. The error approximation is most likely to be valid for problems that are almost rectangular or with mildly varying coefficients, but only for small shape parameter values. %In Section~\ref{sec:exp}, we evaluate its performance for the third test problem with curved boundaries. 

\begin{figure}
\centering
\includegraphics[width=0.47\textwidth]{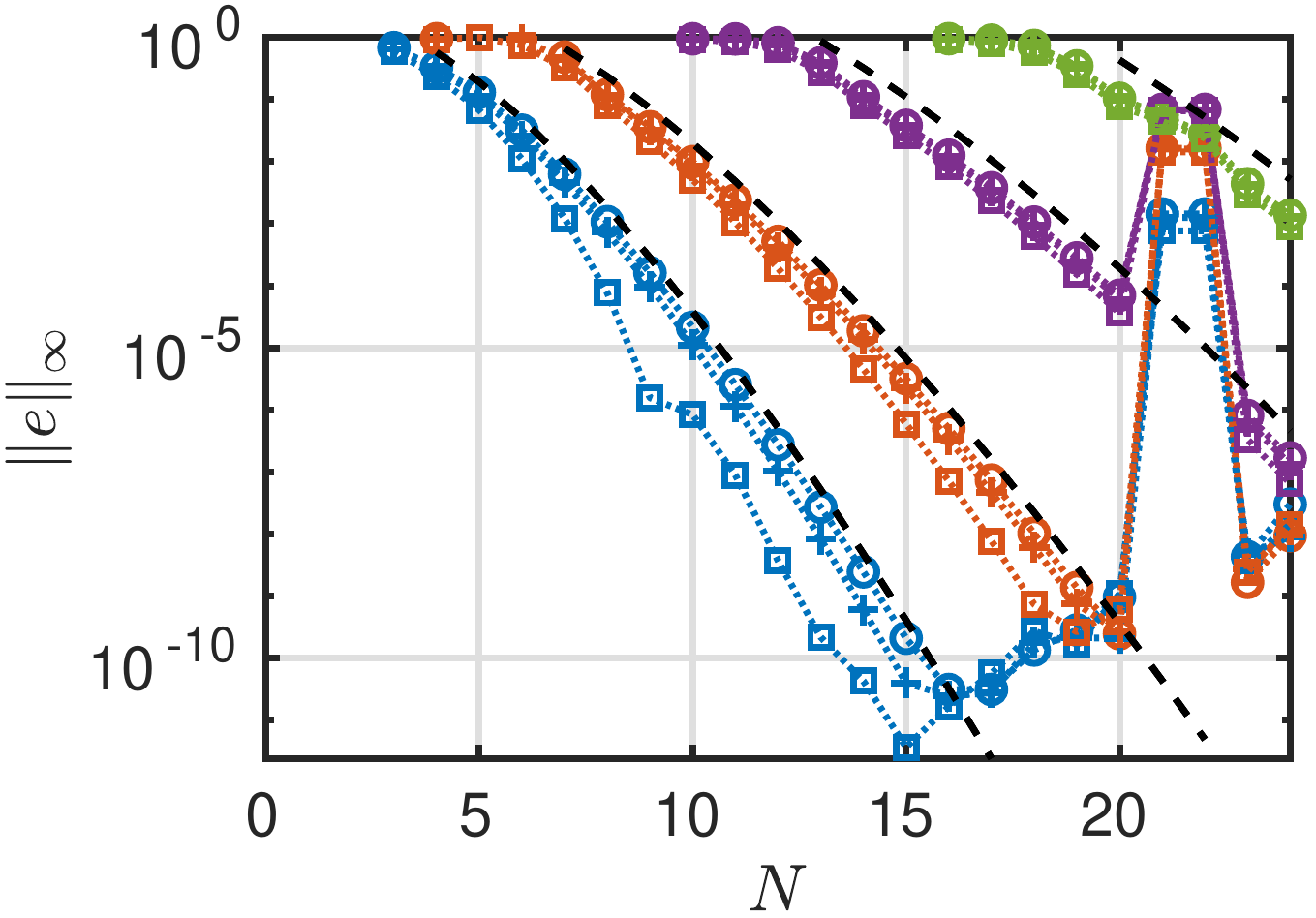}
% 1/40 and 1/800 used as constant for the two highest curves.
\includegraphics[width=0.47\textwidth]{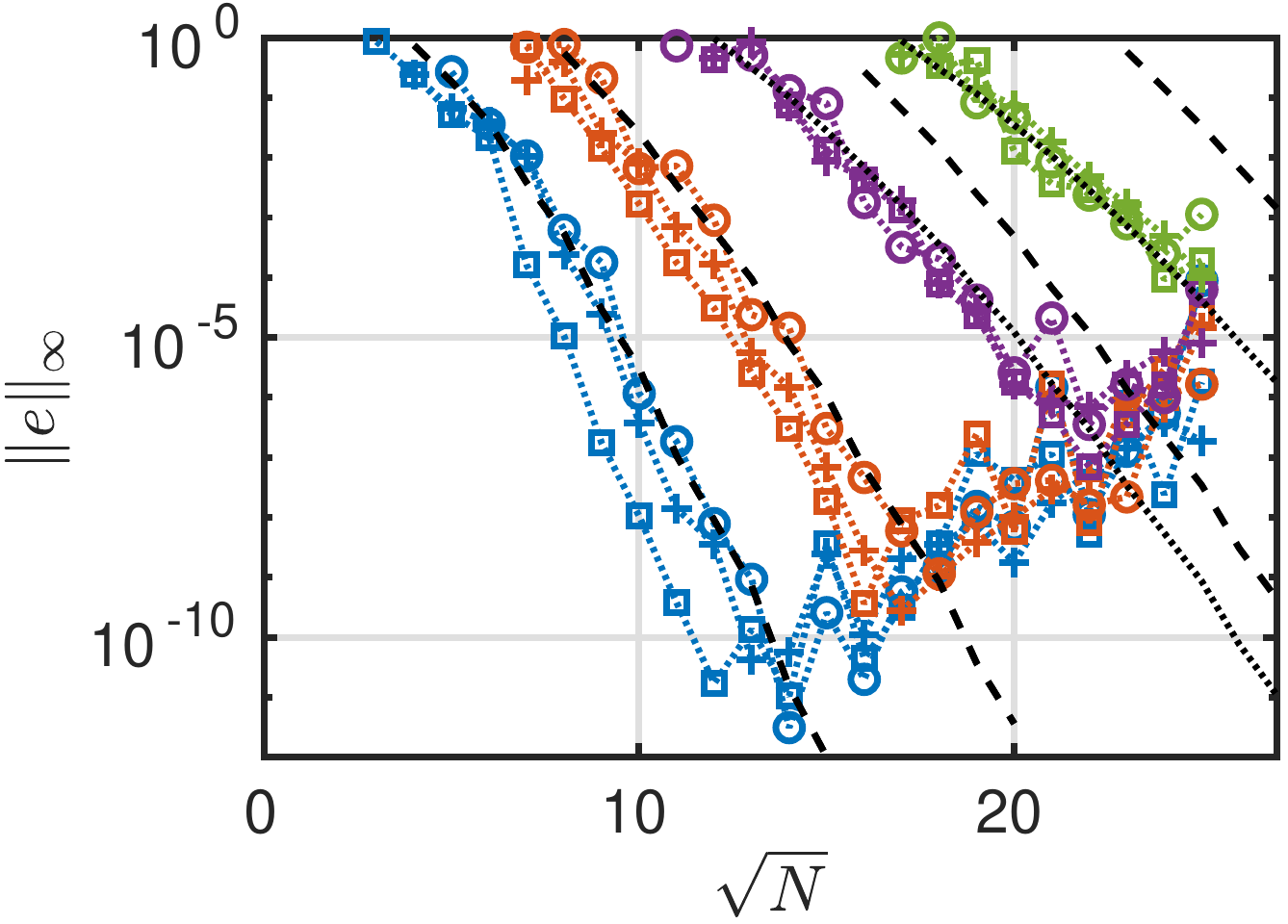}
\caption{The computed errors using the Gaussian RBF and the RBF-QR method for $\ep{}=0.5$ (\protect\scalebox{0.8}{$\Box$}), $\ep{}=0.25$ ($+$), and $\ep{}=0.01$ ($\circ$) for $\kappa=\pi,\,2\pi,\,4\pi,\,6\pi$, from
left to right, together with the approximation $\|e\|_\infty\approx \frac{1}{2}(\kappa h)^{N-1}$ for the one-dimensional problem (left), and for $\kappa=1.2\pi,\,2.4\pi,\,4.8\pi,\,7.2\pi$, from
left to right, together with the approximation $\|e\|_\infty\approx (\kappa h)^{K}$ for the two-dimensional problem (right) (dashed curves). For $\kappa=4.8\pi$ and $7.2\pi$ in the two-dimensional case, we also show the error approximation using $C=1/40$ and $C=1/800$, respectively (dotted lines).}
\label{fig:smallep}
\end{figure}

%%%%%%%% TODO: If there is resonance for even values of W, for the second problem, it should be stated clearly.
\newpage
\subsection{Convergence properties for larger $\ep{}$}
% Compact cube, does it have any derivatives. We can check what the other papers say also. No derivatives, but the estimate is consistent. Different for Gaussian and IMQ. What have we use? k^(k/2) GS, k^k IMQ Gives log(h)/h for Gaussian and C/h for IMQ.
As shown in~\cite{Schaback07,Schaback16,LaShchHe17}, the convergence of a PDE approximation can be expressed in terms of the approximation properties of the interpolant (consistency error) and a stability term. The consistency error of the PDE operator can for example be expressed as
\[\mathcal{E}_\mathcal{L}=\mathcal{L}(I_h(u)-u),\]
where $I_h(u)$ interpolates $u$ using a node set with fill distance $h$. Several authors have derived exponentially converging error results for RBF interpolation~\cite{Powell91,MadNel92,Madych92,BuhDyn93,WuScha93,RieZwi10}. The first papers are focused on interpolation errors, while~\cite{RieZwi10} also provides estimates for derivatives of functions with many zeros, such as the interpolation error. We use the optimality property of RBF interpolants in the native space (reproducing kernel Hilbert space)~\cite{Fass07}
\[\|I_h(u)\|_{\mathcal{N}(\Omega)}\leq \|u\|_{\mathcal{N}(\Omega)},\]
to replace the interpolation error norm with the function norm, since $\|I_h(u)-u\|_{\mathcal{N}(\Omega)}=\|\mathcal{E_I}\|_{\mathcal{N}(\Omega)}\leq 2\|u\|_{\mathcal{N}(\Omega)}$.
We get the following estimates for RBF interpolants in compact cube domains using~\cite[Corollary 5.1]{RieZwi10} for Gaussians
%Compact cubes for Gaussians, convert to N-dependence?
\[\|\mathcal{E_I}\|_\infty\leq e^{C_G\log(h)/h}\|\mathcal{E_I}\|_{\mathcal{N_G}(\Omega)}\leq
2e^{C_G\log(h)/h}\|u\|_{\mathcal{N_G}(\Omega)},\]
% In this case, log e = Clog h / h + b. That is, a comparison should be against log h/h, not 1/h.
where $C_G>0$, and for inverse multiquadrics
\[\|\mathcal{E_I}\|_\infty\leq e^{-C_{Q}/h}\|\mathcal{E_I}\|_{\mathcal{N_Q}(\Omega)}\leq 2e^{-C_{Q}/h}\|u\|_{\mathcal{N_Q}(\Omega)},\]
when $h\leq h_0$, and with $C_Q>0$. This is the same $h_0$ as in the sampling inequality~\eqref{eq:sampl}, which means that the condition is restrictive.
The constants $C_{G}$ and $C_{Q}$ depend on the number of dimensions $d$ and properties of the domain $\Omega$.

The results for derivatives of the interpolation error are given for Lipschitz domains, which are more general than compact cubes, but the results are instead weaker in terms of the convergence rate. From~\cite[Theorem 3.5]{RieZwi10}, we get
\[\|\mathcal{E_L}\|_\infty\leq 2e^{\tilde{C}_{G}\log(h)/\sqrt{h}}\|u\|_{\mathcal{N_G}(\Omega)},\]
\[\|\mathcal{E_L}\|_\infty\leq 2e^{-\tilde{C}_{Q}/\sqrt{h}}\|u\|_{\mathcal{N_Q}(\Omega)},\]
for Gaussians and inverse multiquadrics respectively. The higher rate of Gaussian RBFs is related to the behavior of embedding constants for the native space in relation to Sobolev spaces of increasing order. The constants $\tilde{C}_{G}$ and $\tilde{C}_{Q}$ depend on properties of the domain $\Omega$, and $\tilde{C}_{Q}$ also depends on $\mathcal{L}$ and $d$. In~\cite{RieZwi14}, it is shown that the better convergence rates are obtained also for derivatives of the interpolation error if the nodes are clustered in a layer close to the boundary.

In order to investigate numerically what the actual behavior of the error is for the Helmholtz problem, we solve the one-dimensional problem for a range of shape parameter values and different numbers of node points. In this test, we have used multiquadric RBFs. We assume that the error for multiquadric RBFs has the form
\[\|e\|_\infty=A_M\exp(-C_Mf(h)),\]
where $C_M>0$, $f(h)=1/h$ or $f(h)=1/\sqrt{h}$, and the native space norm has been absorbed into the constant. Then a plot of the logarithm of the error against $f(h)$ should result in a straight line. From Figure~\ref{fig:expconv}, it is clear that $f(h)=1/h$ is a better fit. The dashed lines correspond to a fit of the model with $f(h)=1/h$ to the actual errors, where the results suffering from ill-conditioning effects have been ignored.
\begin{figure}[!htb]
\centering
\includegraphics[width=0.47\textwidth]{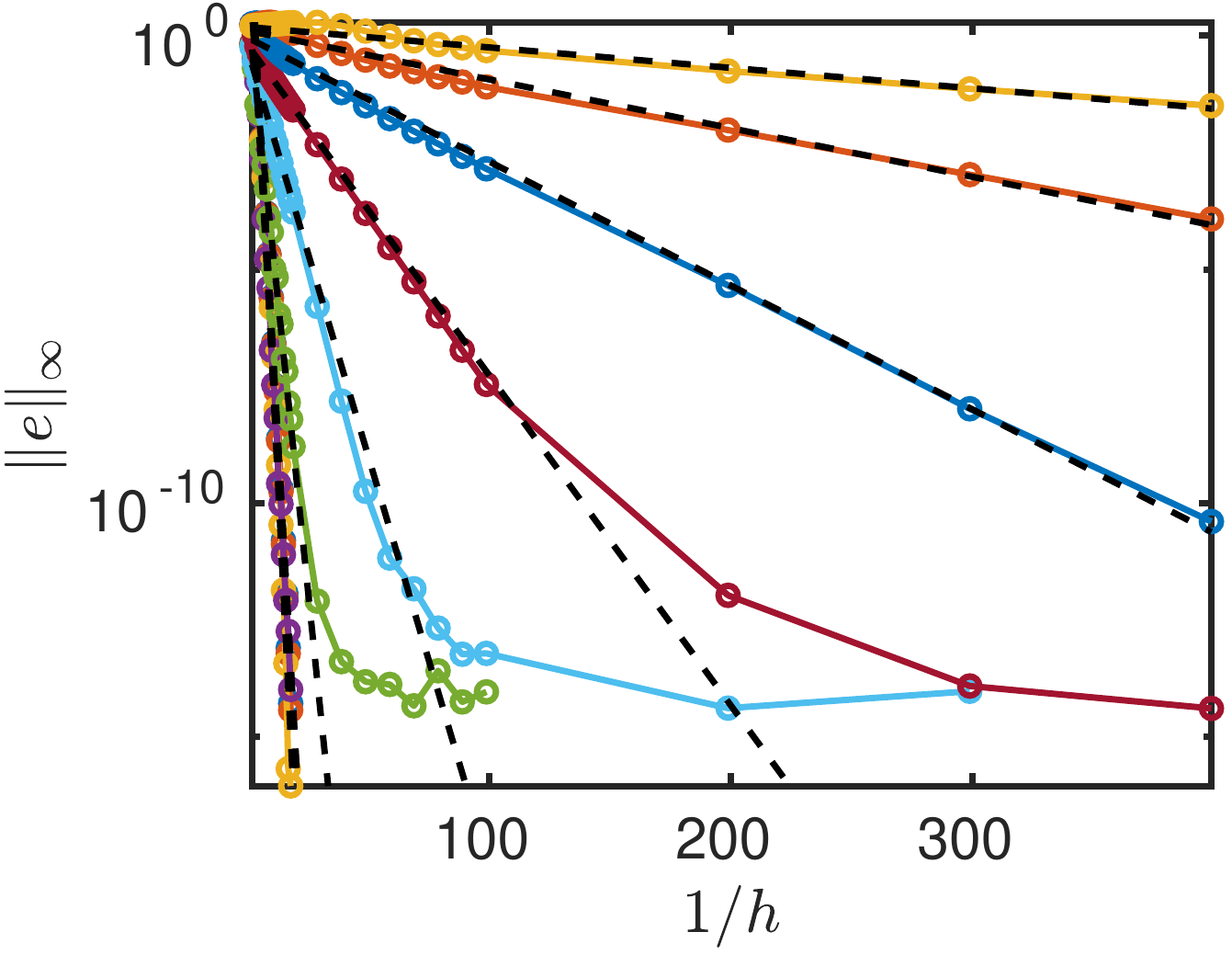}
\includegraphics[width=0.47\textwidth]{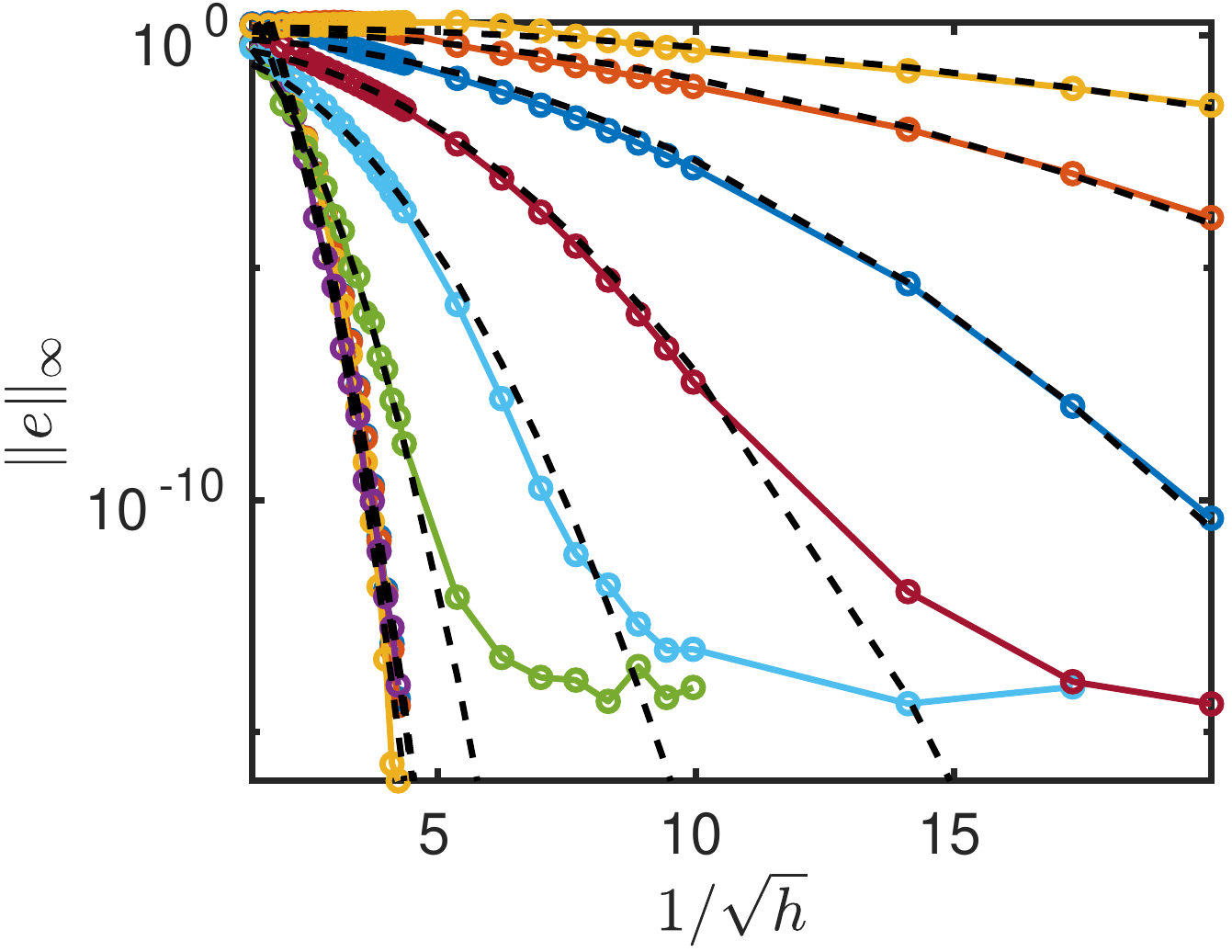}
\caption{The error in the one-dimensional Helmholtz solution when multiquadric RBFs are used as a function of $1/h$ (left) and $1/\sqrt{h}$ (right) for shape parameters $\ep{}=10^{-2+\frac{4}{9}q}$, $q=1,\ldots,9$ (left to right). The dashed black lines/curves correspond to a fit of $\|e\|_\infty=A_M\exp(-C_M/h)$ to the error data (in both cases).}
\label{fig:expconv}
\end{figure}

Figure~\ref{fig:consterr} shows the fitted model parameters $A_M$ and $C_M$ for different shape parameter values. The different curves correspond to different wavenumbers, and it should be noted that the exponential rate $C_M$ becomes independent of the wave number when $\ep{}\gtrapprox0.5$. The rate also decreases with increasing shape parameter values. The optimal rate is attained for a small positive shape parameter value, and for even smaller shape parameters, the asymptotic (polynomial) rate is dominating. The coefficient $A_M$ instead seems to be largest where the rate is optimal, and smallest where the rate is lowest, which makes it harder to determine the best shape parameter value. We discuss this further in the following subsection.
\begin{figure}[!htb]
\centering
\includegraphics[width=0.47\textwidth]{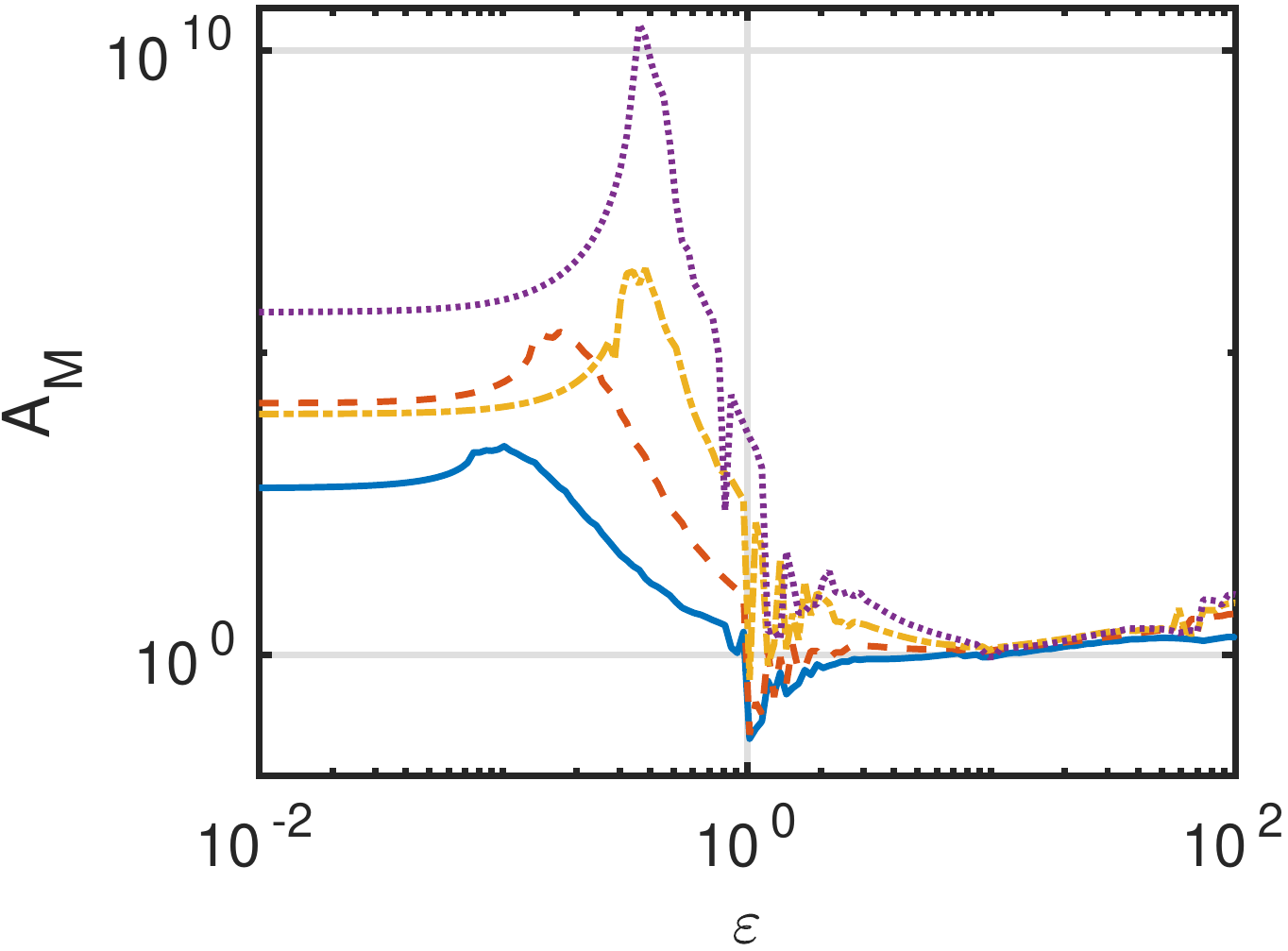}
\includegraphics[width=0.47\textwidth]{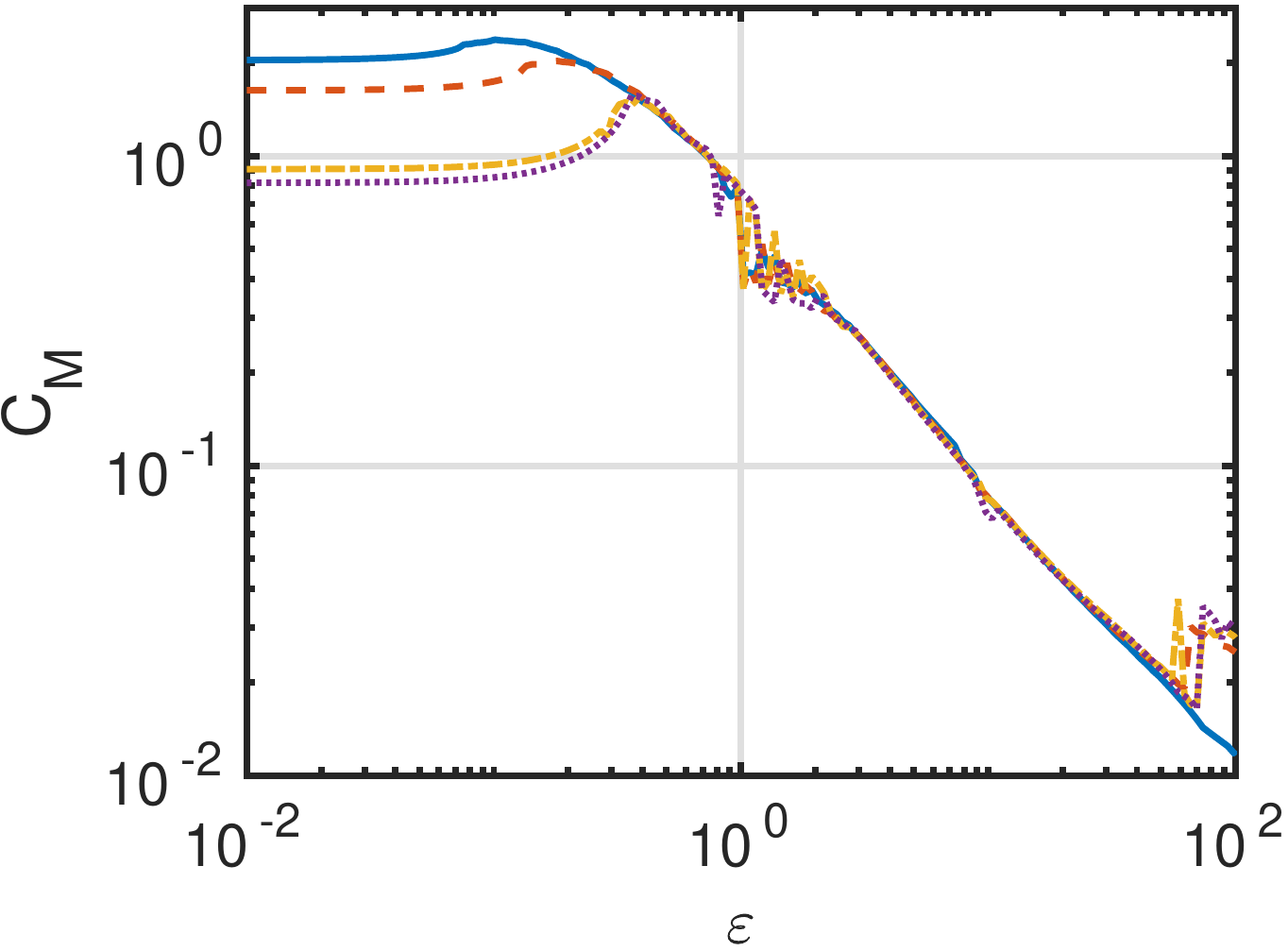}
\caption{The result of fitting the model parameters $A_M$ and $C_M$ to the computed errors for the one-dimensional Helmholtz problem using multiquadric RBFs and different values of the shape parameter $\ep{}$, and for $\kappa=\pi$ (solid line), $\kappa=2\pi$ (dashed line), $\kappa=4\pi$ (dash-dot line), and $\kappa=6\pi$ (dotted line).}
\label{fig:consterr}
\end{figure}

\subsection{Convergence as a function of the shape parameter}
\label{sec:epdep}
Dependence on the shape parameter is not discussed in~\cite{RieZwi10}, and the results reported in the previous subsection hold for a fixed value of $\ep{}$. However, using a shape parameter $\ep{}_0\neq 1$ for an interpolation problem defined in the domain $\Omega$ with fill distance $h$ is equivalent to using a shape parameter $\ep{}_1=1$ for a problem in the scaled domain $\ep{}_0\Omega$ with fill distance $\ep{}_0h$. This can be understood by noting that $\phi(\ep{}_0\|\underline{x}_i-\underline{x}_j\|)=\phi(1\cdot\|\ep{}_0\underline{x}_i-\ep{}_0\underline{x}_j\|)$. Hence, the native space norm is the same in both cases, and the errors are the same in both cases.

If we let the constants $A_M$ and $C_M$ in the error estimate for a specific domain $\Omega$ and shape parameter $\ep{}$ be denoted by $A_M(\Omega,\ep{})$ and $C_M(\Omega,\ep{})$, this means that
\begin{equation}
  A_M(\Omega,\ep{})e^{-C_M(\Omega,\ep{})/h}=A_M(\ep{}\Omega,1)e^{-C_M(\ep{}\Omega,1)/(\ep{}h)}.
  \label{eq:eph}
\end{equation}
That is, the convergence rate for a fixed value of $\ep{}$ is increasing for smaller values of $\ep{}$. This can also be seen in Figure~\ref{fig:consterr}, where the slope in the logarithmic plot of $C_M$ against $\ep{}$ is approximately equal to $-1$. It should be stressed that this does not hold in the flat limit regime, only for $\ep{}\gtrapprox0.5$ (in our case). This also corresponds to the theoretical result given in~\cite{Madych92}, where an explicit constraint on the smallest shape parameter for which the results hold is given as $\ep{}\geq 1/D$, where for a cube domain, $D$ is the side. This coincides well with the numerical results. However, there is also an upper bound $\ep{}\leq 1$, which is harder to reconciliate with what we observe.  

% Two more things. The spiderman plots, and what is said about optimal shape parameters in Theoretical...
% 1. Forensics, which basis function and which kappa for each case. What does thory say about convergence? What is the modified problem? Maybe Bengt and Natasha's paper answer this question.
% Also the error as a function of epsilon
% Perhaps Sloodan and Lina and previous paper explain this well also.

Figures~\ref{fig:errors} and~\ref{fig:errors2} show the error as a function of $\ep{}$ for two one-dimensional problems, and one two-dimensional problem, respectively. The error curves represent a common behavior for smooth solution functions. Starting from a large shape parameter and moving towards smaller values, the error first decreases rapidly then reaches an optimal region, and finally levels out at the polynomial approximation error, see~\cite{LaFo05} for a more detailed discussion about the error curve and the optimal shape parameter.

% This is for non-uniform points. I could solve the linear system for the limit poly for 4 points. Still quite good.
%ep*^2=8/6u_2/u_0 + sqrt[(8/6 u_2/u_0)^2 + 8/3u_4/u_0]

%exp(i\kappa x) = 1+(i kappa x) + 1/2()^2+1/6()^3+ 1/24()^4

%8/6*1/2()^2 + sqrt( 8/6*1/2()^2 + 8/3*1/24()^4 )

%-4/6 k^2 + sqrt(16/36k^4 + 8/3 1/24k^4) = (-4/6 + sqrt(5)/3)k^2 \approx 0.079 k^2

%-u_4/u_2 = 2*kappa^4/24/kappa^2=kappa^2/12

%%%%%%%%%%%%%%%%%%%% Kör lösaren för 4 punkter och testa felet för olika ep%%%%
% Vill veta om vi är nära optimum.

Due to the conditioning problems for decreasing values of $\ep{}$ and increasing values of $N$, a common approach in the literature is to scale the shape parameter such that, e.g., $\ep{}{h}=C$, which is called stationary interpolation. A problem is that stationary interpolation does not converge as $h$ goes to zero. This can be understood by again looking at analogous problems. I we start from a problem on the domain $\Omega$ with shape parameter $\ep{}$ and fill distance $h$, and we refine to get fill distance $h/q$ and shape parameter $q\ep{}$, then the equivalent problem is $(q\Omega,\ep{},h)$. That is, the refinement corresponds to stretching out the domain, while keeping the fill distance and shape parameter constant. This makes the apparent solution function become increasingly smooth, and approaching a constant. Since constants are only reproduced for $\ep{}=0$ for the commonly used infinitely smooth RBFs, there is no convergence for a fixed non-zero $\ep{}$. By augmenting the RBF approximation with polynomial terms, convergence corresponding to the polynomial order can be recovered also in the stationary case~\cite{FFBB16}. 

The convergence curves when choosing the shape parameter as $\ep{}=Ch^\beta$ for different exponents $\beta$ are also shown as dashed lines in Figures~\ref{fig:errors} and~\ref{fig:errors2}. As expected, the stationary choice, $\beta=-1$ levels out as $N$ increases. For $\beta>-1$ we get convergence along different paths. The choice $\beta=0$ corresponds to the exponential convergence case for fixed shape parameter values. For these Helmholtz problems, the curve with $\ep{}=Ch^{3/2}$ captures the optimal shape parameter values well. For other types of problems, the relation would look different.
%
%The maximum error for different choices of the shape parameter $\ep{}$ and 
%the number of node points $N$, for two different one-dimensional Helmholtz 
%problems is shown in Figure~\ref{fig:errors}. We will use these figures as
%starting point for a discussion of convergence rates. 
\begin{figure}[!htb]
\centering
\includegraphics[width=0.47\textwidth]{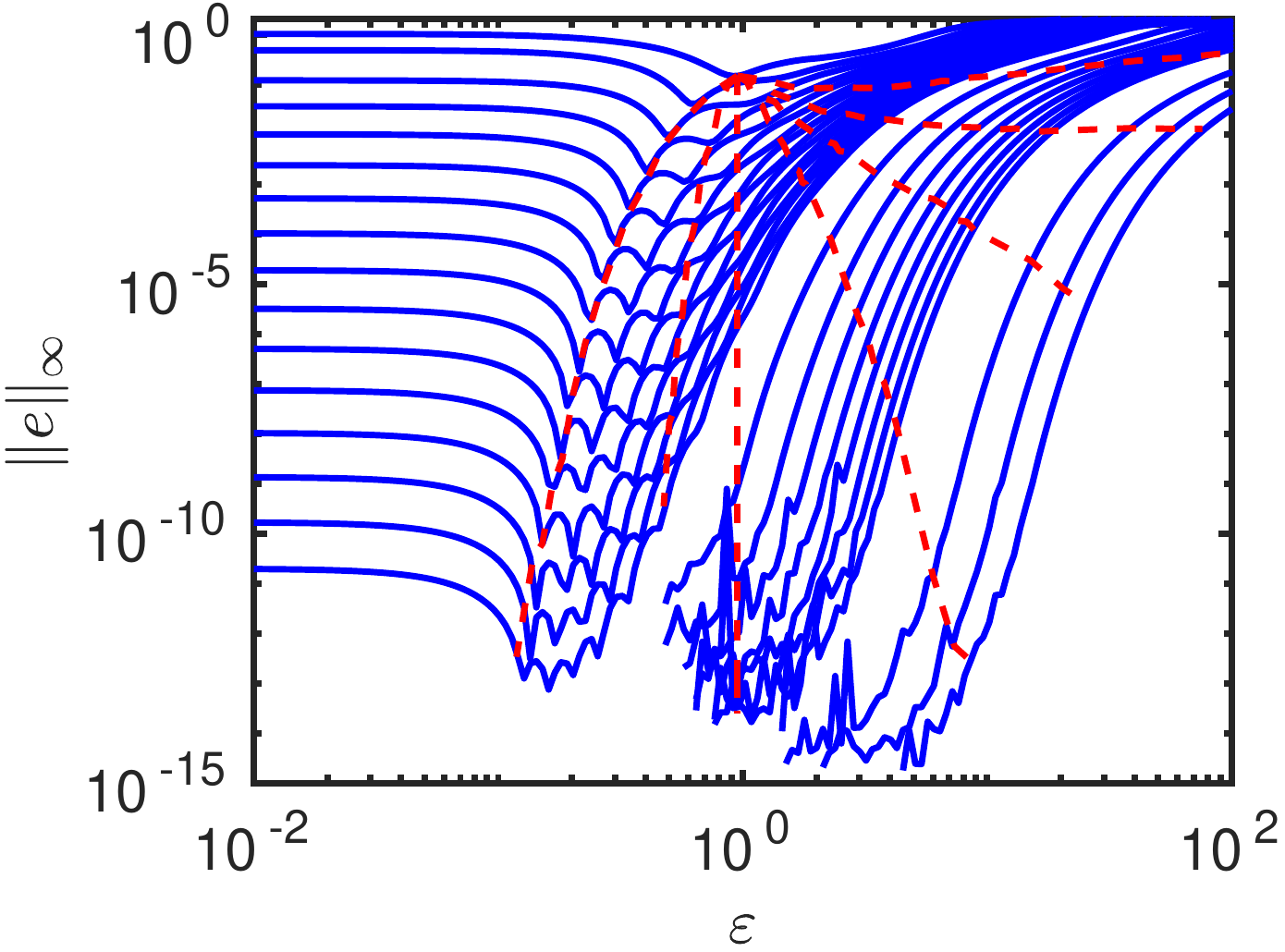}
\includegraphics[width=0.47\textwidth]{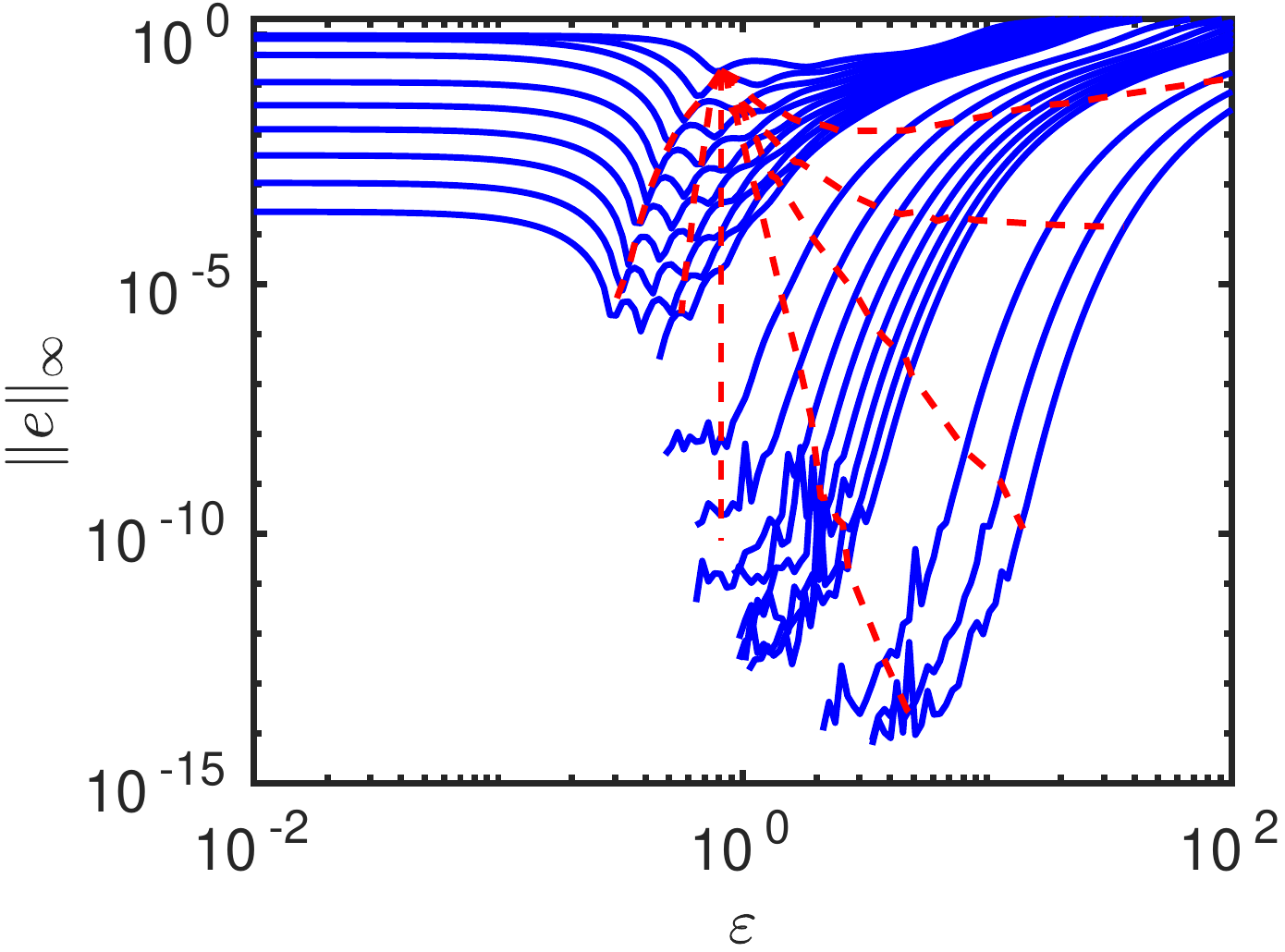}
\caption{The maximum error as a function of $\ep{}$ for $\kappa=2\pi$ (left)
and $\kappa=4\pi$ (right) using multiquadric RBFs. The number of node points is from top to bottom $N=6,\,7,\ldots,21,\,30,\,40,\ldots,100,\,200,\,300,\,400$ in the left subfigure, and
$N=10,\,11,\ldots,20,\,30,\ldots,100,\,200,\,300,\,400$ in the right subfigure. The dashed lines show how the error curves are traversed if the shape parameter is chosen as 
$\ep{}=Ch^\beta$, with $\beta=\frac{3}{2},\frac{1}{2},0,-\frac{1}{2},-\frac{3}{4},-1,-\frac{3}{2}$ from left to right.}
\label{fig:errors}
\end{figure}

For the two-dimensional problem, the curves are more irregular due to several interacting terms in the error~\cite{LaFo05}. However, the overall behavior for the different ways to choose the shape parameter is very similar to the one-dimensional case.

%
% Can I easily add some errors? Run RBF_DtN for more node points.
%

\begin{figure}
\centering
\includegraphics[width=0.5\textwidth]{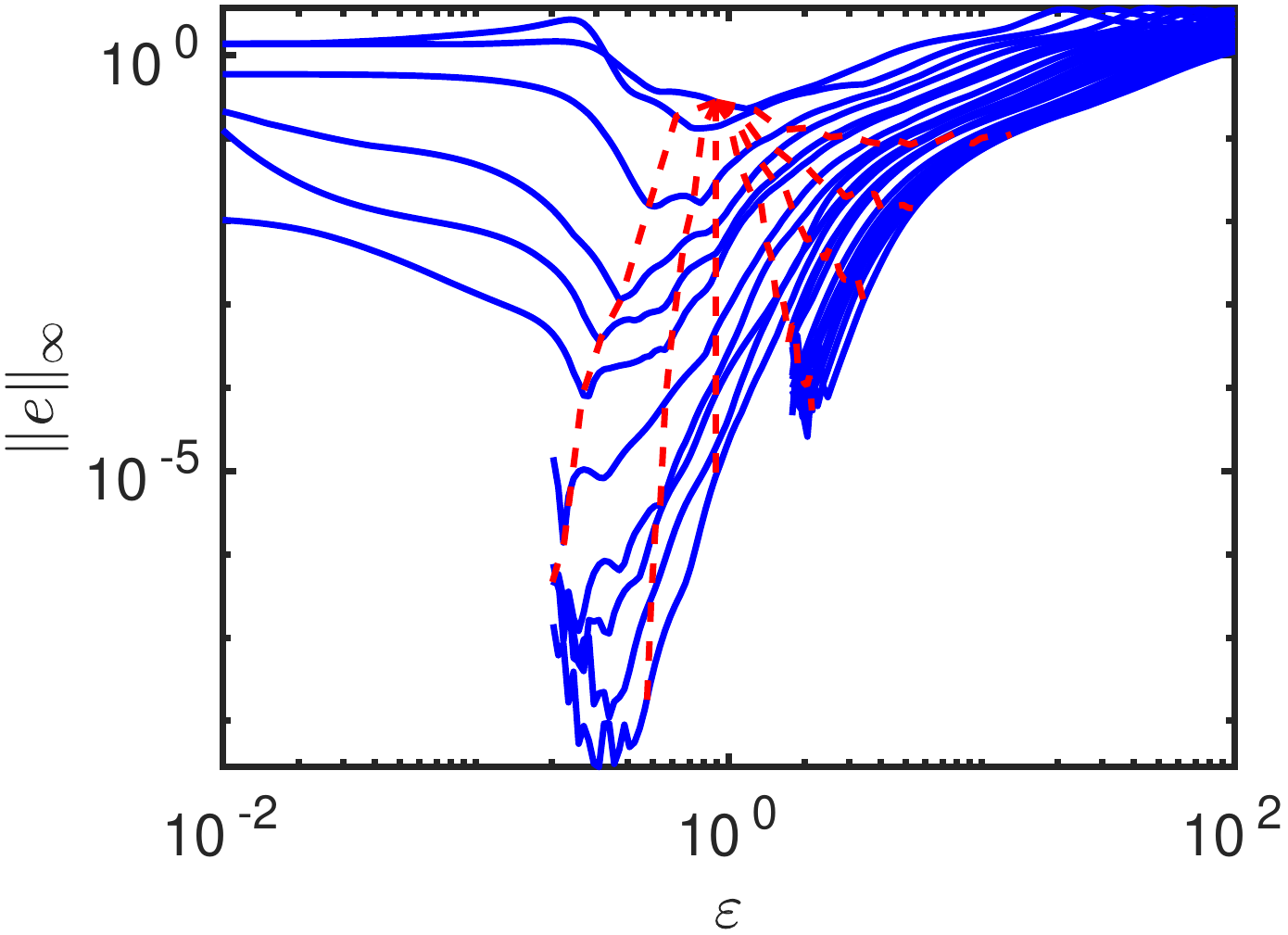}
\caption{The maximum error as a function of $\ep{}$ for $\kappa=2.2\pi$ for the two-dimensional problem using multiquadric RBFs. The number of node points is from top to bottom
  $N\approx n^2$, for $n=3,\ldots,25$.
  %9,\,16,\,25,\,37,\,49,\,66,\,81,\,100,\,122,\,144,\,169,\,196,\,227,\,256,\,290,\,325,\,362,\,401,\,443,\,484,\,529,\,577,\,626$.
  The dashed lines show how the error curves are traversed if the shape parameter is chosen as 
$\ep{}=Ch^\beta$, with $\beta=\frac{3}{2},\frac{1}{2},0,-\frac{1}{2},-\frac{3}{4},-1,-\frac{3}{2}$ from left to right.}
\label{fig:errors2}
\end{figure}
% NOTE: Not deterministic how many point we get.
Assuming that $C_M(\ep{}\Omega,1)$ in~\eqref{eq:eph} does not vary strongly with $\ep{}$, something that can be verified by noting that the slope the line in Figure~\ref{fig:consterr} for $C_M$ is approximately equal to $-1$, we can finally provide a convergence rate for the scaled $\ep{}$ convergence case. If we have exponential convergence as $1/\ep{}h$ and $\ep{}=Ch^\beta$ we end up with
\begin{equation}
  \|e\|_\infty= A_M^{\ep{}} e^{-C_M^{\ep{}}/h^{\beta+1}},\quad -1<\beta\leq 0,
  \label{eq:ephscale}
\end{equation}
where $C_M^{\ep{}}>0$ and the superscript indicates the potential $\ep{}$-dependence. If $\beta>0$, the convergence curves may enter the polynomial region, and we cannot in general get increasing convergence rates for increasing $\beta$. 
The validity of this is expression is further investigated numerically in Section~\ref{sec:exp}.

\section{Numerical experiments}\label{sec:exp}
% In this section, we focus on the third problem, see Figure 1.
In this section, we focus on the third test problem with curved boundaries, see Figure~\ref{fig:1}. We look at how to choose the method parameters and how we can use the theoretical estimates to interpret the results. Unless otherwise mentioned, the problem parameters are given by wavenumber $\kappa=6\pi$, source location $x_s=0.3$, and boundary curves
\begin{eqnarray*}
  \gamma_1&=&0.3\exp(-20(x_2-0.5)^2),\\
  \gamma_2&=&0.8 - 0.3\left(\exp(-80(x_2-0.3)^2) + \exp(-80(x_2-0.7)^2)\right).
\end{eqnarray*}  

For global RBF approximations and shape parameters that are not in the flat limit a uniform node spacing is in general recommended~\cite{PlaDri05}. However, when the problem size is large enough, there can instead be problems at the boundaries unless the nodes are clustered towards the boundaries~\cite{PlaTreKu11,FoLaFly11}. In our experiments, we do not reach the regime where this is an issue. Therefore, we use quasi uniform nodes. The nodes are constructed from the input parameters $n_1$ and $n_2$, that specify the number of nodes in the vertical direction at the left boundary, and the number of nodes in the horizontal direction. We define the step sizes $h_1=0.8/(n_1-1)$ and $h_2=1/(n_2-1)$. Based on these step sizes, the nodes are then placed uniformly along vertical lines with as similar node distance as possible. The nodes at the top and bottom boundaries are placed with uniform arc length. If the nodes are too regular, they are not unisolvent, and the conditioning gets higher at least for shape parameters that are small~\cite{LaFo05}. Therefore, we add a random perturbation to each node. In all experiments performed here, the size of the random perturbation is $0.25(h_1,h_2)$ for the interior nodes, while boundary nodes are only perturbed along the boundary. The solution, residual, and errors are evaluated on a grid. An example of both nodes and evaluation grid is given in Figure~\ref{fig:nodes}.
\begin{figure}
  \includegraphics[width=0.47\textwidth]{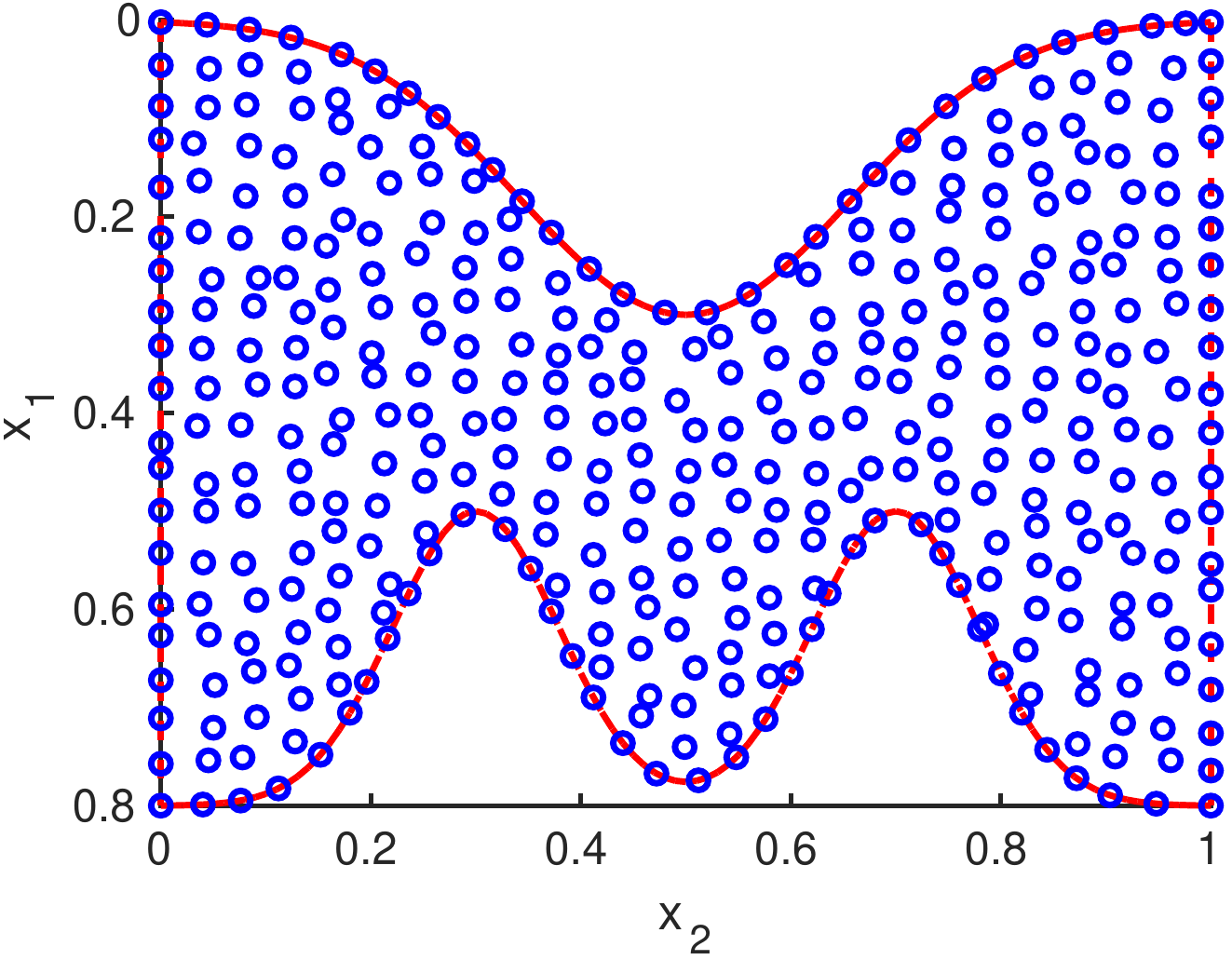}
  \includegraphics[width=0.47\textwidth]{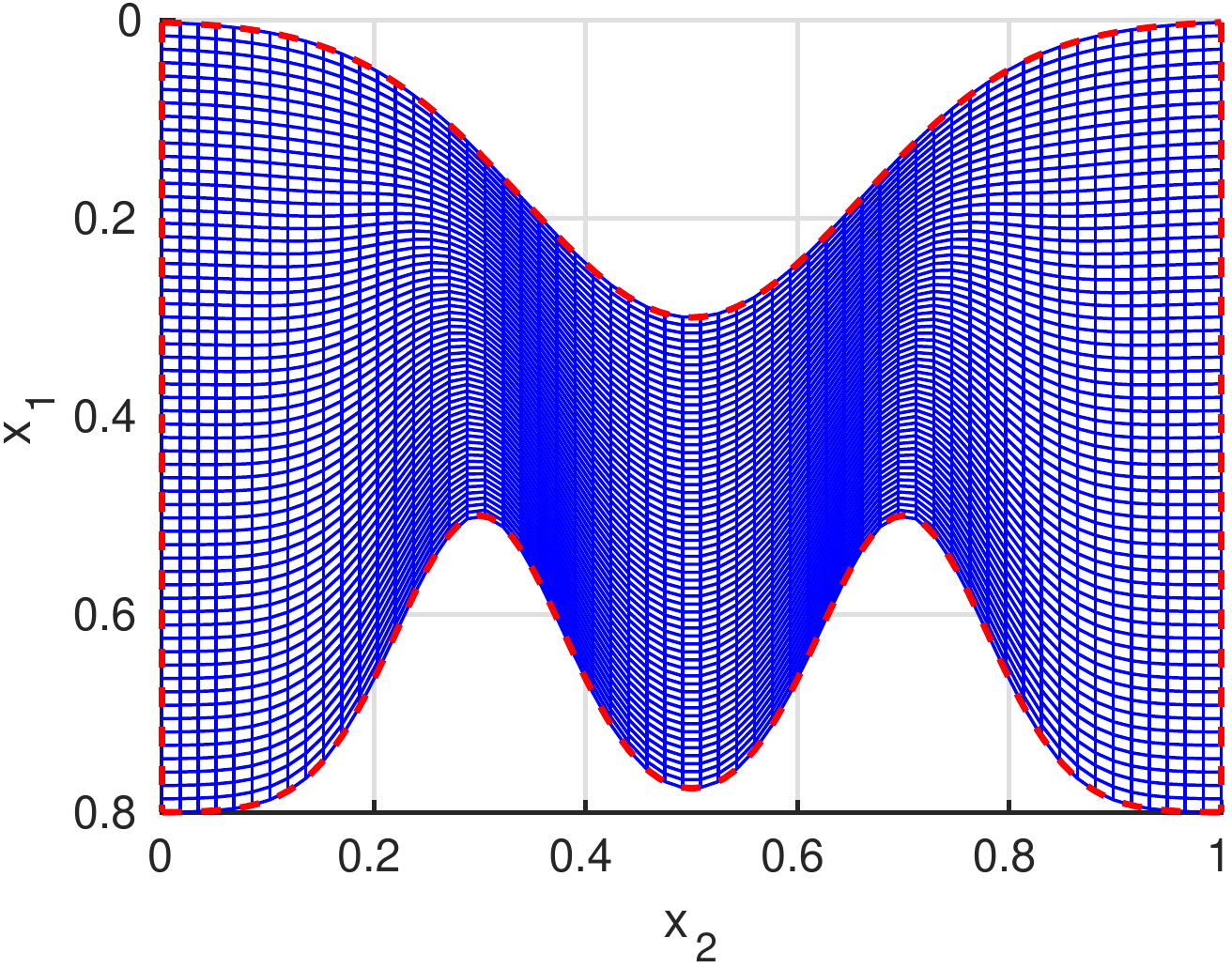}
  \caption{Node points with $n_1=20$ and $n_2=25$ (left) and the evaluation grid with $60\times 60$ points used for the convergence experiments (right).}
  \label{fig:nodes}
\end{figure}  
The resulting numbers of node points for the grids we have used in the experiments are shown in Table~\ref{tab:nodes}.
\begin{table}
  \centering
  \caption{The size $N$ of the different node sets that are used in the experiments. The parameters $n_1$ and $n_2$ are chosen to make $h_1$ and $h_2$ as equal as possible.}
  \label{tab:nodes}
  \small
\begin{tabular}{|c|ccccccc|}\hline
  $n_1\times n_2$ & $10\times12$ &  $11\times14$ &  $12\times15$ &  $13\times16$ &  $14\times17$ &  $15\times19$ &  $16\times20$\\\hline
  $N$            & 104 & 131 & 152 & 174 & 194 & 235 & 261\\\hline
  \multicolumn{8}{c}{\hspace{0mm} }\\\hline  
  $n_1\times n_2$ & $17\times21$ &  $18\times22$ &  $19\times24$ &  $20\times25$ &  $22\times27$ &  $24\times30$ &  $26\times32$\\\hline
  $N$            & 287 & 317 & 362 & 396 & 462 & 563 & 639\\\hline
  \multicolumn{8}{c}{\hspace{0mm} }\\\hline  
  $n_1\times n_2$ & $28\times35$ &  $30\times37$ &  $32\times40$ &  $34\times42$ &  $36\times45$ &  $38\times47$ &  $40\times50$\\\hline
    $N$            & 747 & 844 & 971 & 1079 & 1219 & 1341 & 1493\\\hline
  \multicolumn{8}{c}{\hspace{0mm} }\\\hline  
  $n_1\times n_2$ & $50\times62$ &  $60\times75$ &  $70\times87$ &  $80\times100$ &  $90\times112$ &  $100\times125$ &  \\\hline
  $N$            & 2294 & 3306 & 4434 & 5813 & 7300 & 9029 & \\\hline
\end{tabular}  
\end{table}

Errors are measured against a reference solution computed using the largest node set with $n_1\times n_2=100\times125$. This is the largest problem size that fits in the memory of the Dell Latitude E6230 laptop with an i5-3360 dual core CPU running at 2.8 GHz that was used for the experiments. When we refer to the maximum norm of the numerical errors or the solution, we evaluate them on the $60\times 60$ evaluation grid, except for the solutions with higher wavenumbers, where we use $100\times 100$ grid points. We use multiquadric RBFs in all numerical experiments. The MATLAB implementations of the solvers that were used in the experiments are available at the first authors software page.
%%%%%%%%%%%%%%% I think this should be moved to a numerical section, where we explain the problem a bit more and compute mu0 etc.

% 1. Jämför function evaluations med qr-fallet. Kanske behöver göra om de andra då jag ändrat från quadl. Problem that I am using mq, when qr only does Gaussians.
% 1b: Just make the same table for Gaussians

% The table is for 30*38 pts, mq, W=3.

%
% Can I do the qr-problem with clustering?
%

\subsection{Selecting a tolerance for constructing the DtN boundary conditions}
As mentioned in Section~\ref{sec:non-symm}, we need to compute $N$ inner products with each vertical eigenmode $\psi_m$ present in the problem at the two vertical boundaries. Accurate numerical computation of these integrals is a significant computational cost, e.g, up to $n_f=1700$ function evaluations per integral are needed for tolerance $1e-15$. The question is which tolerance to choose.

The sensitivity of the problem (ill-conditioning) depends strongly on the shape parameter $\ep{}$ with an exponentially increasing condition number as the shape parameter goes to zero. By using a stable evaluation method such as RBF-QR for Gaussian RBFs, the sensitivity is removed and the tolerance for the integrals does not need to be smaller than the desired error in the solution. However, for the test problem considered here, too small values of $\ep{}$, leading to a global polynomial approximation is not an appropriate choice, and we are not able to use RBF-QR.

Table~\ref{tab:1} shows the
average number of function evaluations needed by MATLAB's \texttt{quadl} to
approximate one integral to a prescribed absolute tolerance for different values of
the shape parameter $\ep{}$ for a node set with $n_1\times n_2=30\times 38$.
%for the problem in
%Figure~\ref{fig:1}.
The bold faced entries in the table show the largest
tolerance that can be used before the approximation changes significantly. The tolerance is much smaller than the absolute error in the solution, which is about 0.5 compared with the reference solution. The condition numbers computed by MATLAB are between $1\cdot10^{17}$ for $\ep{}=5$ and $1\cdot 10^{11}$ for $\ep{}=12$. For larger $N$, the ill-conditioning also increases, so we expect that even smaller tolerances are needed in this case.   
\begin{table}[!htb]
\centering
\caption{The average number of function evaluations for approximating one integral of the type in 
~(\ref{eq:DtN}) using \texttt{quadl} for multiquadric RBFs. Bold faced numbers show the largest
tolerance that does not significantly alter the result. The relative error against the reference solution is also given. A $\times$ indicates that the approximation had an error of the same order as the size of the solution.}
\begin{tabular}{|c|cc|cc|cc|cc|}\hline
Tolerance & \multicolumn{2}{c|}{1e$-$4} & \multicolumn{2}{c|}{1e$-$6} & \multicolumn{2}{c|}{1e$-$8} & \multicolumn{2}{c|}{1e$-$10}\\\hline% & 1e$-$12\\\hline
$\ep{}=5$  &33 & $\times$   & 52&$\times$  &  97&$\times$  & \textbf{164}&2.5e$-$1\\%  & \textbf{320}\\ 
$\ep{}=6$  &34& $\times$  & 54&$\times$  & \textbf{102}&1.4e$-$1  & 168&1.4e$-$1\\%  & 335\\
$\ep{}=7$  &35& $\times$  & 56&6.3e$-$1  & \textbf{106}&6.1e$-$2   & 173&5.9e$-$2\\%  & 348\\
$\ep{}=8$  &36& $\times$  & 58&6.6e$-$2  & \textbf{109}&5.2e$-$2   & 180&5.2e$-$2\\%  & 360\\
$\ep{}=9$  &36& $\times$  & \textbf{59}&6.0e$-$2   & 112&5.0e$-$2   & 187&5.0e$-$2\\%  & 369\\
$\ep{}=10$ &36& $\times$   & \textbf{60}&5.0e$-$2  & 114&5.0e$-$2   & 193&5.0e$-$2\\%  & 377\\
$\ep{}=11$ &37& 3.5e$-$1  & \textbf{62}&5.1e$-$2  & 115&5.1e$-$2  & 198&5.1e$-$2\\%  & 384\\
$\ep{}=12$ &\textbf{37}& 7.9e$-$2   & 63&5.2e$-$2  & 117&5.2e$-$2   & 202&5.2e$-$2\\\hline%  & 391\\\hline
\end{tabular}
\label{tab:1}
\end{table}

%The same experiment was also performed for Gaussian RBFs and the results are shown in Table~\ref{table:2}. A larger shape parameter is needed to get a solution that is close to the real one. The overall results are a bit worse than for MQ RBFs. [Check abs/rel Update caption]
%\begin{table}[!htb]
%\centering
%\caption{The average cost of evaluating one integral of the type in 
%~(\ref{eq:DtN}) using \texttt{quadl} for Gaussian RBFs. Bold faced numbers show the largest
%tolerance that leads to a stable computation.}
%\begin{tabular}{|l|cc|cc|cc|cc|}\hline
%Tolerance & \multicolumn{2}{c|}{1e$-$4} & \multicolumn{2}{c|}{1e$-$6} & \multicolumn{2}{c|}{1e$-$8} & \multicolumn{2}{c|}{1e$-$10}\\\hline
%$\ep{}=14$  &28& 9.0e$-$1    & 41& $\times$  &  \textbf{61}& 5.7e$-$1 &   105& 5.6e$-$1\\
%$\ep{}=15$  &28& $\times$   & 40& 8.2e$-$1   & \textbf{59}& 2.9e$-$1  & 101& 2.9e$-$1\\
%$\ep{}=16$  &28& $\times$   & 39& 1.9e$-$1  & \textbf{58}& 1.3e$-$1   &  97& 1.3e$-$1\\
%$\ep{}=17$  &27& $\times$   & \textbf{38}& 9.9e$-$2  & 57& 8.8e$-$2  &  93& 8.8e$-$2\\
%$\ep{}=18$  &27& 2.9e$-$1   & \textbf{37}& 1.2e$-$1  & 55& 1.2e$-$1  &  90& 1.2e$-$1\\
%$\ep{}=19$  &27& 1.9e$-$1   & \textbf{36}& 1.6e$-$1  & 54& 1.6e$-$1 &    87& 1.6e$-$1\\
%$\ep{}=20$ &\textbf{26}& 2.8e$-$1   & 35& 2.9e$-$1  & 53& 2.9e$-$1  &   85& 2.9e$-$1\\\hline
%\end{tabular}
%\label{tab:2}
%\end{table}

\subsection{Choosing the starting value for the shape parameter}
% Consider a small problem, low cost.
% Use Loocv, not good enough moves quite a lot
% Show one plot with minima of the estimator, and erorr against reference.
% For the smaller problem sizes. Make a script. Perhaps the estimator over 3.
% Introduce the pragmatic, but convergent approach.
% Then show a plot of the constant C in ep=C/sqrt(h)
To solve a large scale problem efficiently it pays off to choose the shape parameter carefully, since it does not affect the cost, only the accuracy. As was discussed in Section~\ref{sec:epdep}, a practical way to achieve convergence in spite of the ill-conditioning is to choose the shape parameter as $\ep{}=Ch^\beta$, with $\beta>-1$. We are going to use $\beta=-1/2$, which provides a trade-off between convergence rate and conditioning problems. Then we need to decide which $C$ to use.

Compared with the full solution, it is not so expensive to solve a much less resolved problem a few times for different shape parameters. We want to test if the residual-based error estimate~\eqref{eq:errest2} can help us find the best shape parameter value for such a problem, and from there the $C$ to use. We also try the $\ell_2$-norm of the residual as an indicator, since the residual should be small when the error is small. The maximum norm of the residual was also tested, but did not correlate strongly with the error. Figure~\ref{fig:optep} shows the relative error estimate as well as the relative $\ell_2$-norm of the residual together with the actual error against the reference solution. In the first example, the shape parameter values corresponding to the smallest error estimate, $\ep{\mathrm{est}}$, and the smallest residual norm, $\ep{\mathrm{res}}$, are both close to the actual minimum $\ep{*}$. In the second example, the minimum for the error estimate is a bit higher than the true value. 
\begin{figure}[!htb]
\centering
\includegraphics[width=0.47\textwidth]{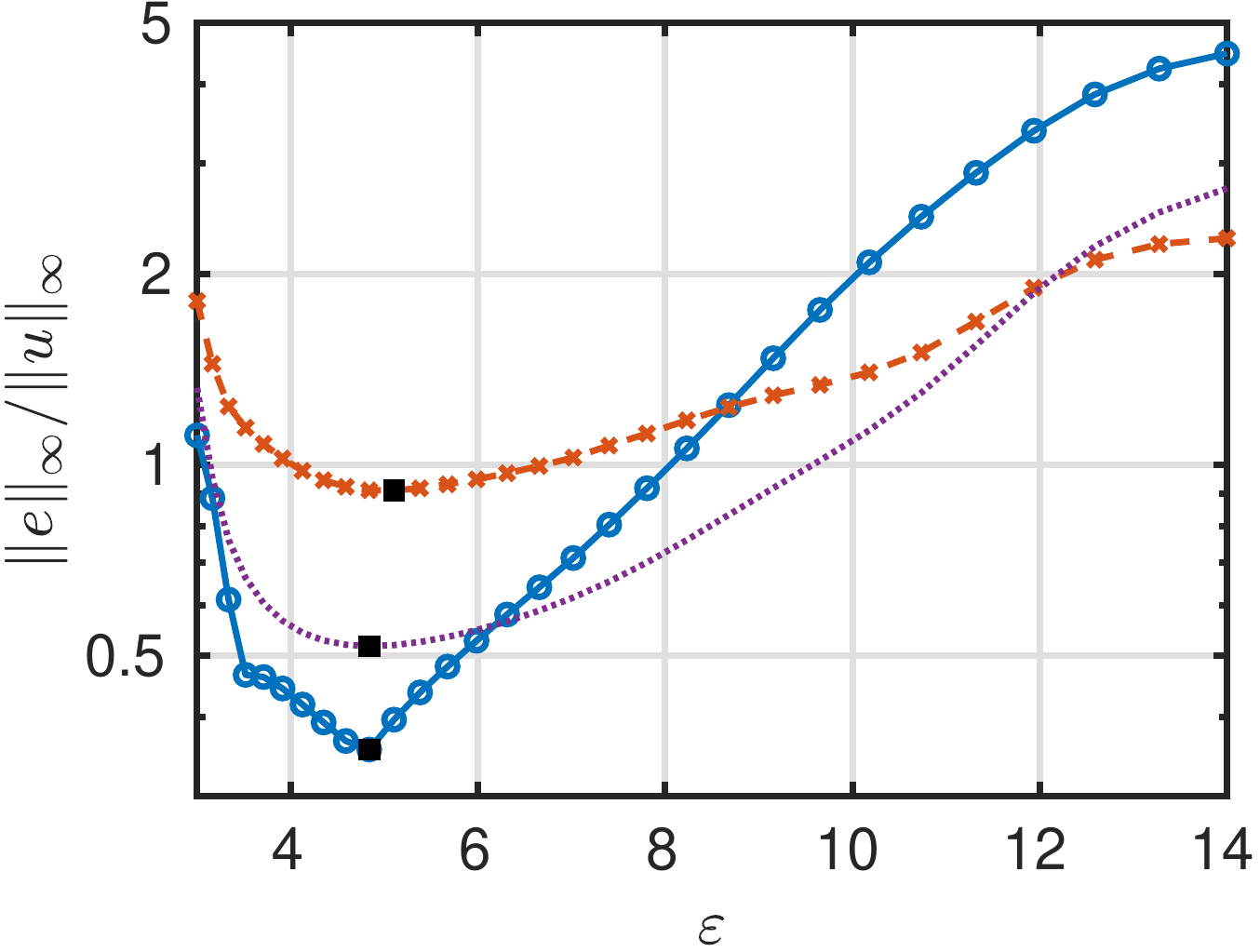}
\raisebox{-0.5mm}{\includegraphics[width=0.5\textwidth]{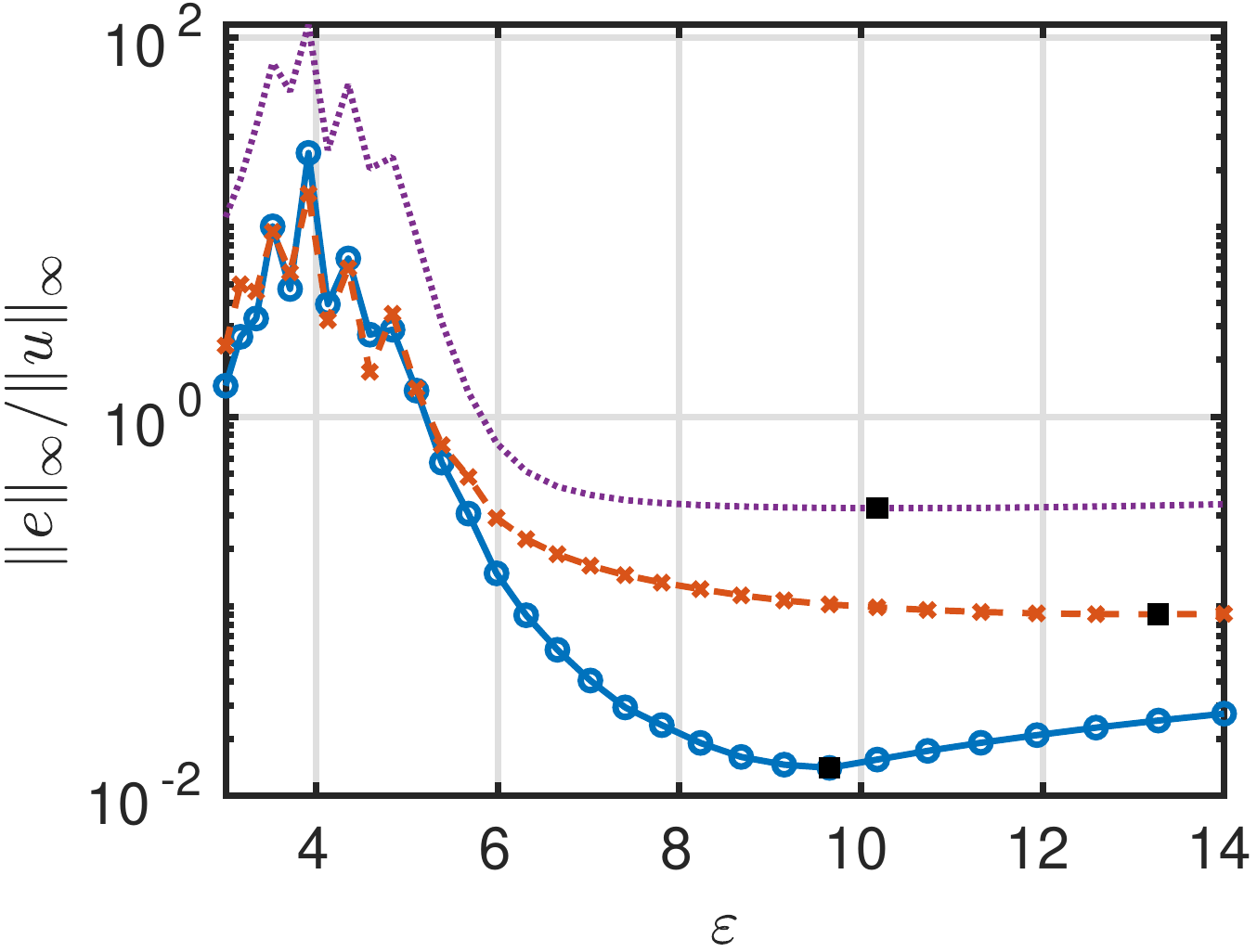}}
\caption{The error estimate~\eqref{eq:errest2} ($\times$), the $\ell_2$-norm of the residual (dotted line) and the error against a highly resolved reference solution ($\circ$) for the $10\times12$ (left) and $40\times50$ (right) node sets.The minima are indicated by black squares.}
\label{fig:optep}
\end{figure}

Table~\ref{tab:epsilon} gives the minimal shape parameter values for ten different (small to medium) problem sizes. In most of the cases the error estimate, the residual estimate, or both are close to the true value. We have also computed the $C$-values corresponding to the average of the two estimates.
\begin{table} % The new version
  \centering
  \caption{The optimal shape parameter for the error against the reference solution $\ep{*}$, the error estimate $\ep{\mathrm{est}}$, and the $\ell_2$-norm of the residual $\ep{\mathrm{res}}$ and the constant $\tilde{C}$ implied by the average of the two estimates for different problem sizes.}
  \label{tab:epsilon}
\begin{tabular}{|l|rrrrr|}\hline
$n_1\times n_2$ & 10 $\times$ 12  & 11 $\times$ 14  & 12 $\times$ 15  & 13 $\times$ 16  & 14 $\times$ 17 \\\hline
$\ep{*}$  &    4.8  &    4.6  &    5.7  &    4.6  &    7.0 \\
$\ep{\mathrm{est}}$   &    5.1  &    6.7  &    8.7  &    6.7  &    3.2 \\
$\ep{\mathrm{res}}$  &    4.8  &    5.4  &    7.4  &    5.4  &    3.7 \\
$\tilde{C}$  &   1.5  &   1.7  &   2.2  &   1.6  &   0.9 \\\hline
$n_1\times n_2$ & 15 $\times$ 19  & 16 $\times$ 20  & 20 $\times$ 25  & 30 $\times$ 37  & 40 $\times$ 50 \\\hline
$\ep{*}$  &    7.4  &    4.8  &    9.7  &   9.2  &   9.7 \\
$\ep{\mathrm{est}}$   &    7.8  &    6.0  &    9.2  &    13.3  &   13.3 \\
$\ep{\mathrm{res}}$  &    6.3  &    5.4  &    7.0  &    9.7  &    10.2 \\
$\tilde{C}$  &   1.7  &   1.3  &   1.7  &   1.9  &   1.7 \\\hline
\end{tabular}  
\end{table}
%For each value of $\ep{}$ 
%\[r_k=\frac{\lambda_k}{L^{-1}_{kk}}\]
If we had solved only the first problem, we would have chosen $C=1.5$. This is what we have used for the convergence experiments in the following subsection.  We also tried $C=1$, but then the ill-conditioning prevented us from solving the largest problems.

An alternative method to find a good shape parameter value is to use the leave-one-out cross validation method. It was first introduced for RBF interpolation methods~\cite{Rippa99}, and a cost effective version of the method was derived in~\cite{YaYaLi18}. It was suggested to use LOOCV for PDE problems using the residual as error indicator in~\cite{ChGoKaZa03}, and this was implemented in~\cite{FeRoJoFa07}. We tried to use residual-based LOOCV on the Helmholtz problems in this paper, but the preliminary results were not close enough to the optimal values, and we therefore decided to use the error approximation instead.

\subsection{Convergence experiments}
Here, we use the relation $\ep{}=C/\sqrt{h}=1.5/\sqrt{h}$ to run a convergence experiment. We solve the test problem for different problem sizes and compute the error estimate and the error against the reference solution. According to equation~\eqref{eq:ephscale}, with this choice of shape parameter scaling, the error should be of the form
\[\|e\|_\infty=A_M\exp(-C_M/\sqrt{h}).\]
In Figure~\ref{fig:errest}, we plot the relative error and the relative error estimate~\eqref{eq:errest2} against $1/\sqrt{h}$. A line has been fitted to the data set, and it is clear from the picture that it is a good fit of the convergence trend. The slopes $C_M$ are 0.78 for the error and 0.75 for the error estimate, which means that the error estimate gives very good results for the ratio of errors at different resolutions, even if the constant is not precise. The constant $A_M$ is 3.0 times larger for the error estimate than for the error. Based on the curves in Figure~\ref{fig:optep}, we expect $A_M$ to be problem and/or parameter dependent.

If we compare the error reduction from the smallest to the largest problem size with what we would get with an algebraically converging method where the error is $\mathcal{O}(h^p)$, a reduction in error with a factor 242 for a step size reduction of 10 corresponds to $p=2.4$. That is, even if we have exponential convergence, the overall error reduction is not that impressive. However, the small numbers of points we can use, while still getting reasonable results are impressive. The smallest problem has 12 points in the horizontal direction, which corresponds to 4 points per wavelength. A rule of thumb for a finite difference method is that at least 15 points per wavelength, that is 45 for this problem, are needed for geometric resolution. 
\begin{figure}[!htb]
\centering
\includegraphics[width=0.47\textwidth]{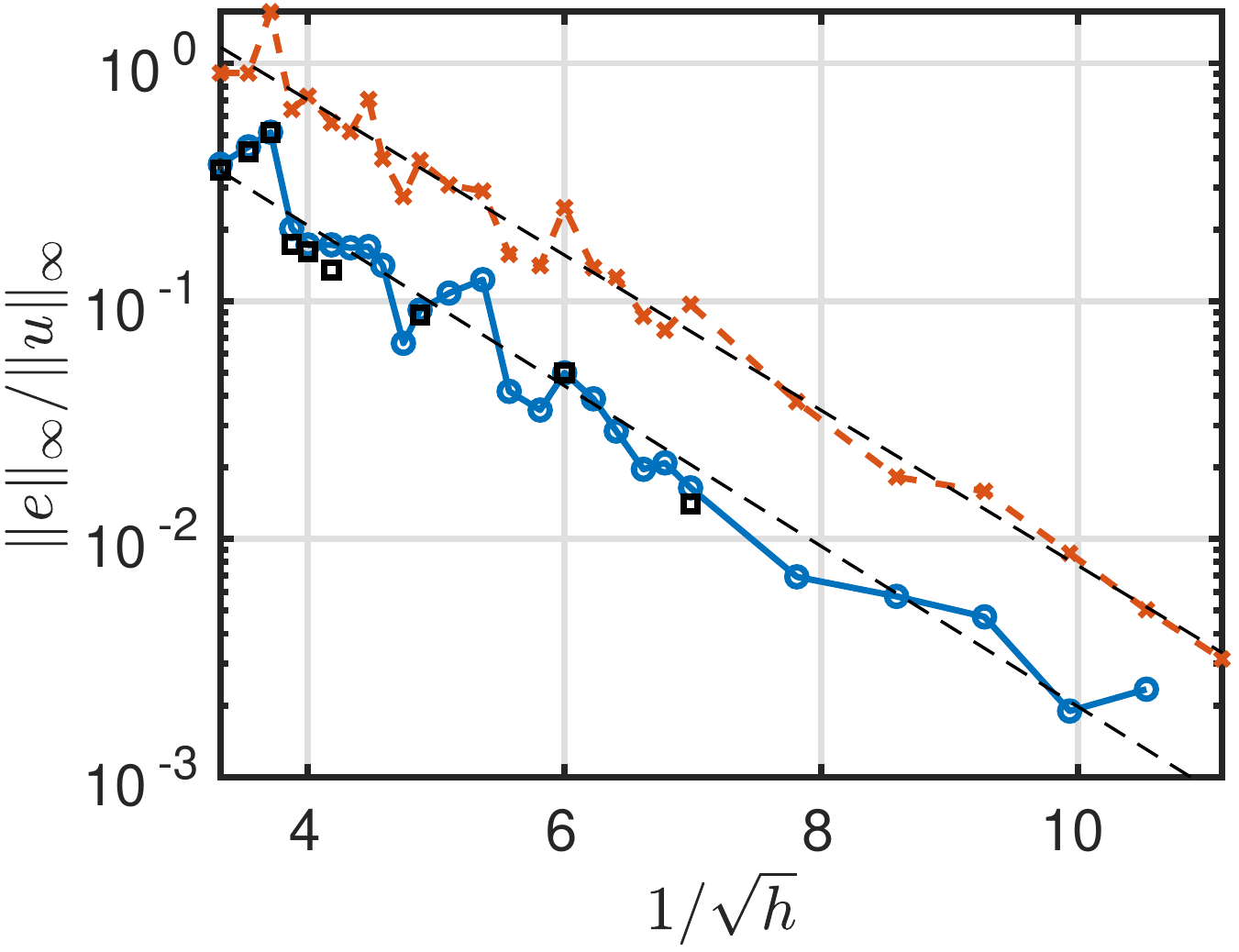}
\includegraphics[width=0.47\textwidth]{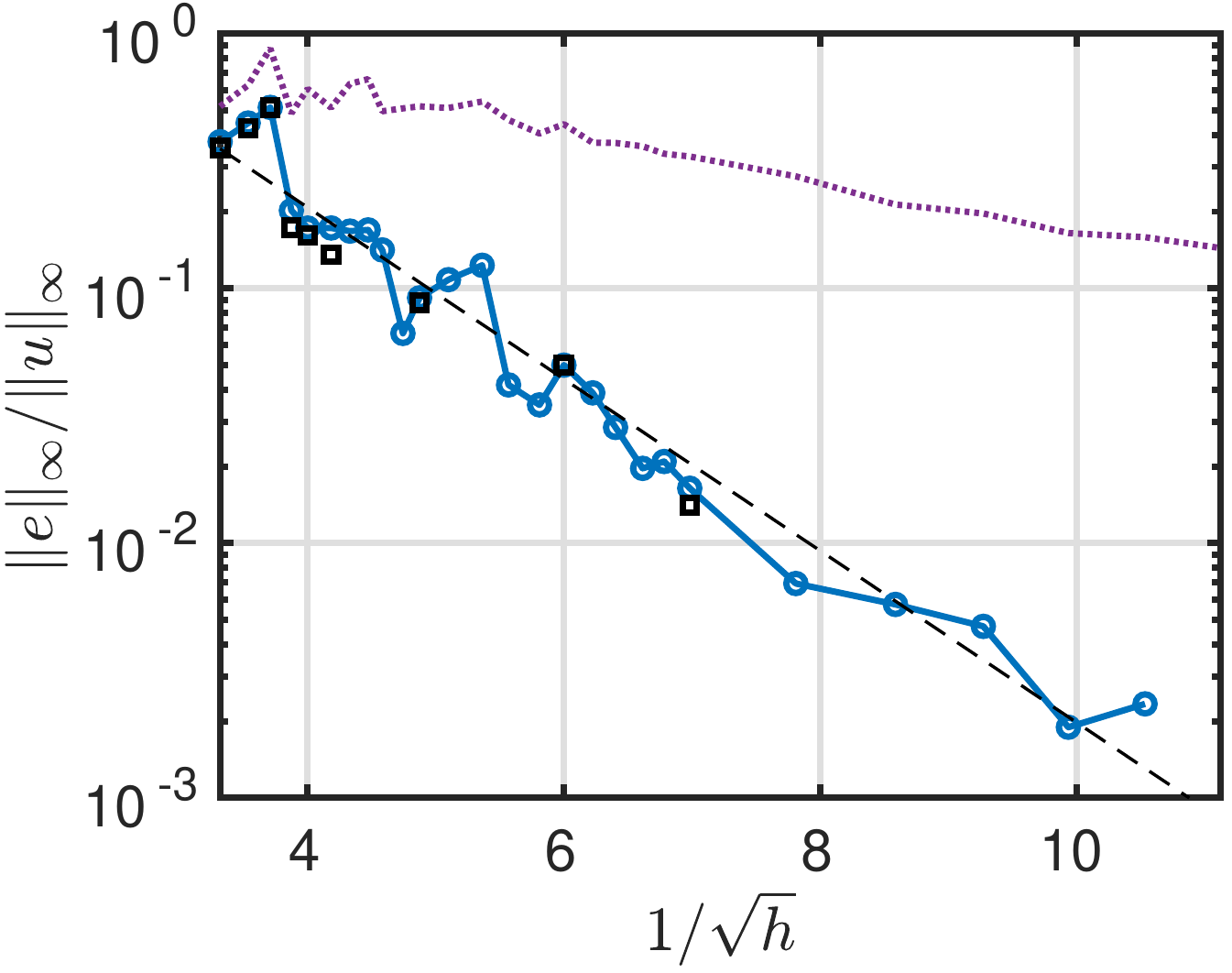}
\caption{The relative error estimate~\eqref{eq:errest2} ($\times$) and the error against the reference solution as a function of $1/\sqrt{h}$ are shown in the left subfigure. The $\ell_2$-norm of the residual is shown together with the same error curve in the right subfigure.
  The black squares are the results for the optimal shape parameter values. The dashed lines represent lines fitted to the data points.}
\label{fig:errest}
\end{figure}
For the largest problems, the tolerance for the quadrature had to be lowered.
%When the condition number of the matrices gets really high, above $1\cdot10^{12}$,
The small perturbations introduced by the inexact quadrature with tolerance $1\cdot10^{-10}$ are enough to prevent the convergence curve from following the straight line, and the convergence rate then seems to decrease. These experiments were run using tolerance $1\cdot 10^{-14}$.

In the right subfigure of Figure~\ref{fig:errest}, the $\ell_2$-norm of the residual is plotted together with the same relative error results. Even though the residual norm gives reasonable estimates for the optimal shape parameter, it is clear that we cannot use it to follow the error trend.
%The situation is quite different when solving the W=1.1 problem in the square with few points and ep=0.01. The error only chcnages in the decimals. does not change and is very small. Can we even have a tolerance larger than the error?

\subsection{Experiments with larger wave numbers}
We have also solved problems with larger wavenumbers as this usually is a challenge for wave propagation problems. For these problems $\kappa=12\pi$ and $24\pi$, corresponding to 6 and 12 wavelengths along the duct. The solution functions are shown in Figure~\ref{fig:sols}.
%
% Look at the step size the error estimate and also compare with the problem we have. Comment on the shape parameter and the results for that.
%
% Create two beautiful pictures to put side by side.
\begin{figure}[!htb]
\centering
\includegraphics[width=0.47\textwidth]{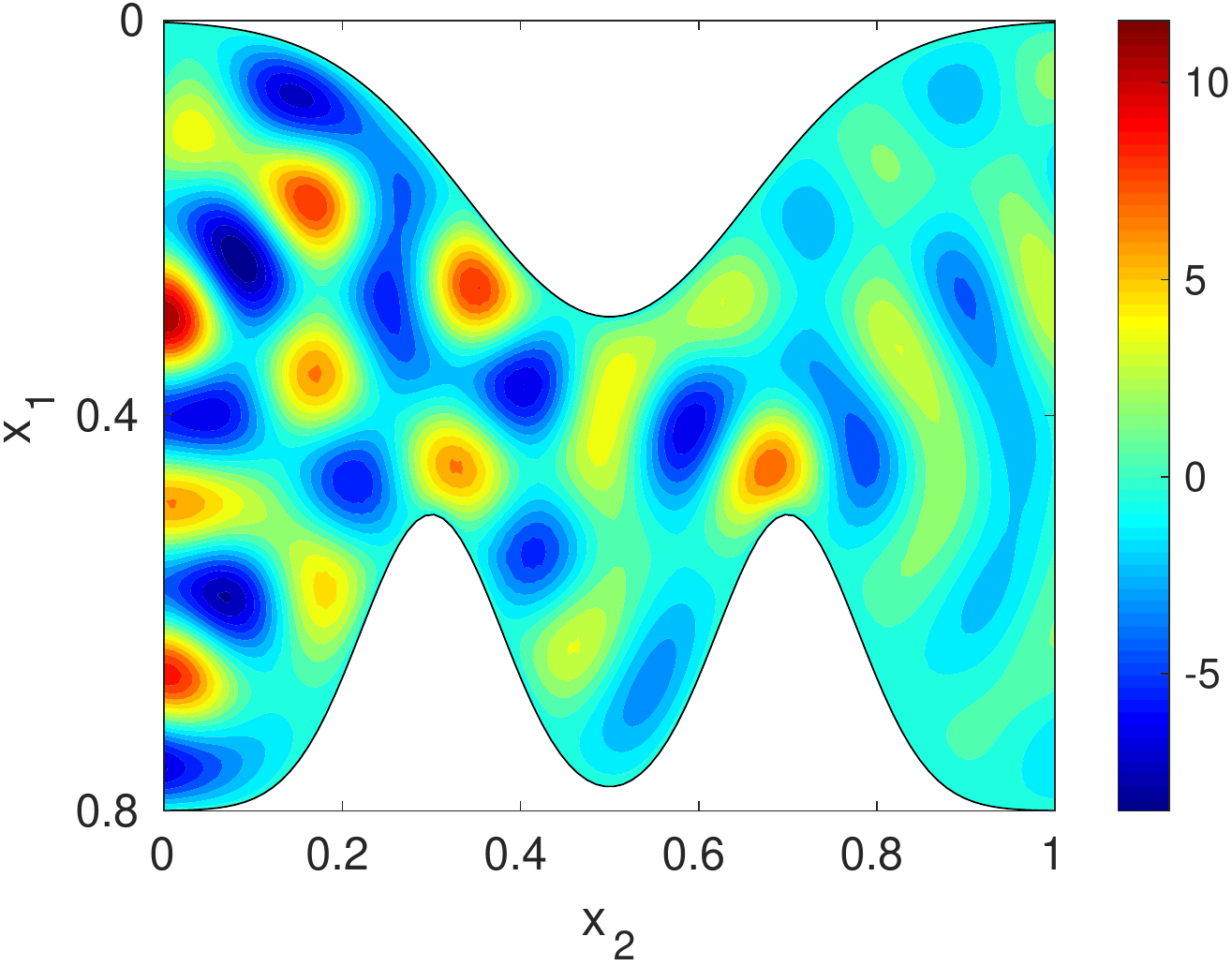}
\raisebox{-0.5mm}{\includegraphics[width=0.47\textwidth]{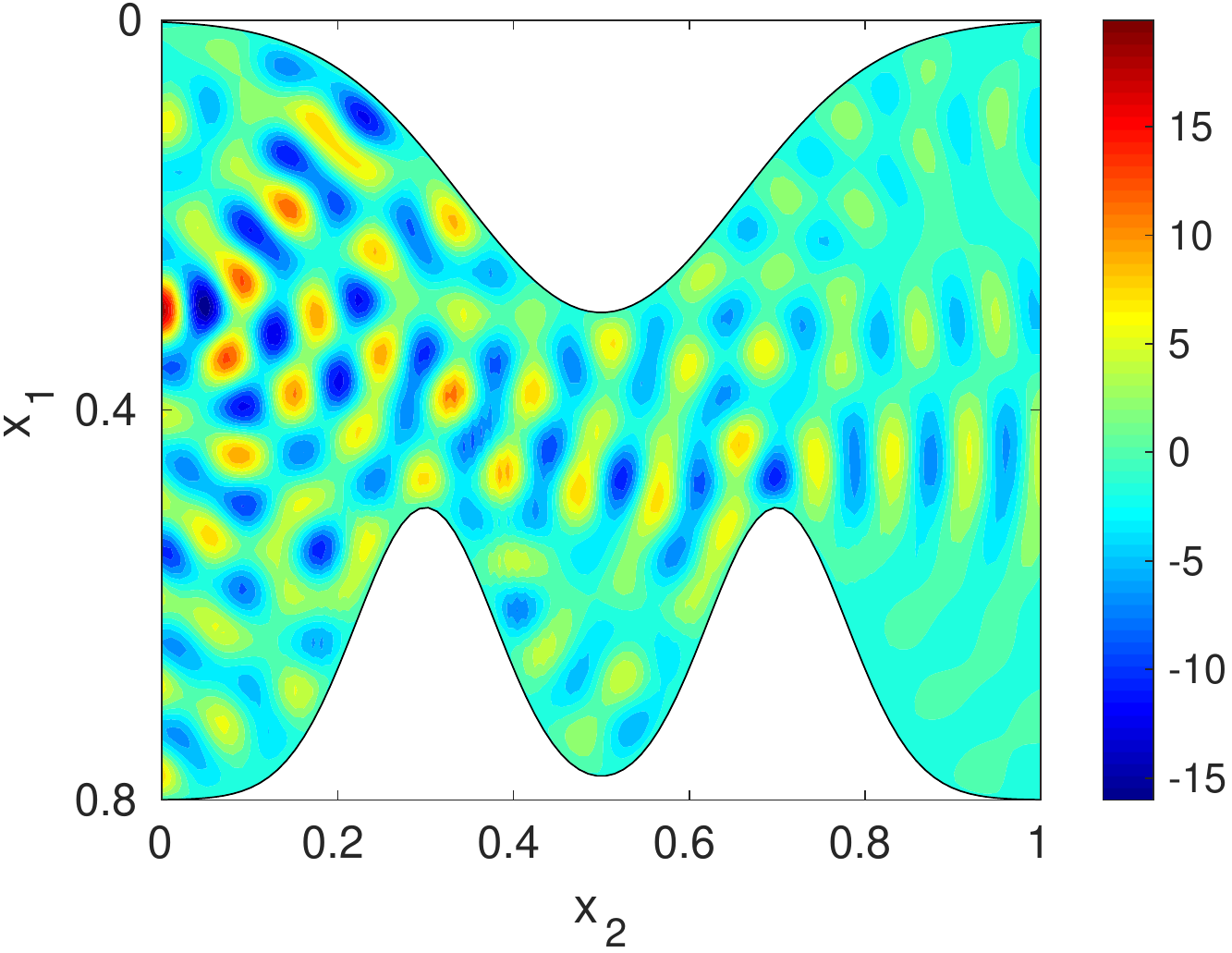}}
\caption{The solution function for $\kappa=12\pi$ (left) and $\kappa=24\pi$ (right). The solution is computed using nodes with $n_1\times n_2=50\times 62$. The source is located at the left boundary at $x_s=0.3$.}
\label{fig:sols}
\end{figure}

% tol is an absolute tol
These solutions have 9 and 19 propagating modes at the left boundary, respectively. A problem here was to compute the inner products with the eigenmodes to high enough accuracy. The accuracy of the boundary conditions is crucial to get the correct wave pattern. We were not able to run the simulations for $\kappa=24\pi$ for a larger problem size than $50\times62$ (with good results). The same shape parameter scheme as for the convergence experiment was used.

For each problem, we ran three different problem sizes in order to get an estimate of the errors in the solutions. Then we computed the relative errors of the coarser solutions with respect to the finest solution. The computed errors are compared with the error estimate~\eqref{eq:errest2} to find the approximate ratio between real errors and estimate. Then we use the worst case ratio to project an error estimate also for the finest solution. The results are shown in Tables~\ref{tab:W6} and~\ref{tab:W12}, indicating around 1\% error for $\kappa=12\pi$ and around 12\% error for $\kappa=24\pi$. The numbers of points per wavelength are 75/6=12.5 and $62/12\approx5.2$, respectively. For the problem with $\kappa=6\pi$ and the same two node sets, we had 0.6--0.7\% error and 21--25 points per wavelength. If we look at the error for 5 points per wavelength for $\kappa=6\pi$, it is around 20\%. That is, it seems that we do not need to resolve more with increasing frequency. For finite difference and finite element methods the error in a waveguide Helmholtz problem is typically proportional to $h^p\kappa^{p+1}$~\cite{BaGoTu85,OttoLa99,Larsson99}. This effect diminishes as the order of the method increases, and for a spectral method it disappears. This is consistent with the results for small $\ep{}$ in Section~\ref{sec:smallep}, where the error approximations are proportional to $(\kappa h)^K$.  
%
% 
% W=6: Relative error against the finest solution 0.0018 0.0033
%      Relative error approximation       0.0699  0.1781 0.3083
%      Factor between the two                        7.4    7.1
%      Adjusted error estimate            0.0099  0.0251 0.0435
%      Slope between the two points            1.2    0.66

% 
% W=12: Relative error against the finest solution 0.1292 0.3842
%       Relative error approximation       0.4756  0.7941 1.4909
%       Factor between the two                        6.1    3.9
%       Adjusted error estimate            0.1226  0.2047 0.3842
%                                               0.64    0.62
\begin{table}
  \centering
  \caption{The relative error in relation to the finest solution, the relative error estimate, the ratio, the adjusted error estimate, and the local slope for the problem with $\kappa=12\pi$.}\label{tab:W6}
\begin{tabular}{|c|ccccc|}\hline
  $n_1\times n_2$ & $\|e\|_\infty/\|u\|_\infty$ & $\|\tilde{e}\|_\infty/\|u\|_\infty$  & $\|\tilde{e}\|_\infty/\|e\|_\infty$ & $\frac{\|\tilde{e}\|_\infty}{\|u\|_\infty}/\min\frac{\|\tilde{e}\|_\infty}{\|e\|_\infty}$ & $C_M$\\\hline 
  40 $\times$ 50 & 0.0435 & 0.3083 & 7.1 & 0.0435 & $-$  \\
  50 $\times$ 62 & 0.0240 & 0.1781 & 7.4 & 0.0251 & 0.66 \\
  60 $\times$ 75 & $-$    & 0.0699 &  $-$& 0.0099 & 1.2\\\hline
\end{tabular}  
\end{table}

\begin{table}
  \centering
  \caption{The relative error in relation to the finest solution, the relative error estimate, the ratio, the adjusted error estimate, and the local slope for the problem with $\kappa=24\pi$.}\label{tab:W12}
\begin{tabular}{|c|ccccc|}\hline
  $n_1\times n_2$ & $\|e\|_\infty/\|u\|_\infty$ & $\|\tilde{e}\|_\infty/\|u\|_\infty$  & $\|\tilde{e}\|_\infty/\|e\|_\infty$ & $\frac{\|\tilde{e}\|_\infty}{\|u\|_\infty}/\min\frac{\|\tilde{e}\|_\infty}{\|e\|_\infty}$ & $C_M$\\\hline 
  30 $\times$ 37 & 0.3842 & 1.4909 & 3.9 & 0.3842 & $-$  \\
  40 $\times$ 50 & 0.1292 & 0.7941 & 6.1 & 0.2047 & 0.64 \\
  50 $\times$ 62 & $-$    & 0.4756 &  $-$& 0.1226 & 0.62\\\hline
\end{tabular}
\end{table}

\section{Discussion}\label{sec:end}
% Preconditioning with stencil-based solution could be an idea. Maybe it would be less sensitive to conditioning when there is only multiplication, but there would be two issues 1, how high resolution would be needed in the low order, and?
%We could also use a larger ep, but there is not always a solution. Could one use simplified boundary conditions? That only allow normal modes to go ut.
% To discuss:
% Non-straight boundaries.
% Solvers because cost kills.
% Conditiong problems for small and for large
% Locality
%
% PROBLEM; METHOD; ERRORS, PRACTICALLY
The main benefits with using global RBF methods for solving Helmholtz-type problems are that very few points per wavelength are needed to obtain a qualitatively correct solution, and that
%, as far as we can see in our experiments,
the number of points per wavelength does not need to increase with $\kappa$ (the number of wavelengths). It is also relevant that non-trivial waveguide geometries can be managed easily, since there is no need for an orthogonal or even a structured grid. In~\cite{OttoLa99,Larsson99}, we used orthogonal grids, which limits how much the boundaries can vary. It should be mentioned that the DtN boundary conditions assume a smooth continuation with horizontal boundaries outside of the domain. In our example, the derivative of the boundary curves is non-zero at $x_2=0,\,1$, which introduces an error. However, since we got optimal convergence rates in the experiments, these errors are not large enough to influence the results at the level of errors that we could reach.  

The main challenge of using a global RBF method for a PDE problem is the computational cost. In Helmholtz applications it is of interest to solve problems that are large in terms of wavelengths, and therefore require a certain resolution. With a dense linear system, both the storage requirements and the computational cost for a direct solver quickly become difficult to manage at least without using distributed computing. On top of that, the severe ill-conditioning of the linear systems makes them sensitive to numerical errors in the quadrature employed in DtN conditions as well as to rounding errors. An attractive alternative to using global RBF collocation methods is to use localized methods such as RBF-generated finite differences (RBF-FD)~\cite{FoFly15} and RBF partition of unity methods (RBF-PUM)~\cite{LaShchHe17}. In~\cite{ShchLa16} it was shown that for an option pricing application, there was no significant difference in accuracy between the global method and  RBF-PUM for a given problem size, while the computational cost is significantly lower for RBF-PUM due to sparsity of the linear systems.

We compared the non-symmetric and symmetric collocation approaches and found that the symmetric method, even though elegant, becomes cumbersome especially for non-trivial operators. The main benefit of the symmetric collocation is the guaranteed non-singularity of the interpolation matrix. However, for the non-symmetric method, singularity only occurred for wavenumbers that were physically uninteresting or for problems that were numerically unresolved. It seems reasonable that if the continuous problem is well-posed and the discrete problem is resolved enough to be close to the continuous problem, singularity is unlikely, see also~\cite{HonSch01,Schaback16}.

We have also investigated the error behavior as a function of $N$ and $\ep{}$ from different perspectives. Some of this can be explained by the limit behavior. We studied this for interpolation in~\cite{LaFo05}, but here we looked at what is different for PDE problems. If the node set is unisolvent and PDE unisolvent, the RBF approximant has the form $s(\underline{x})=P_K(\underline{x})+\ep{2}P_{K+2}(\underline{x})+\ldots$, where $P_K(\underline{x})$ is the unique polynomial solution of degree $K$ to the PDE problem, and $P_{K+2j}$ have zero PDE residual at the node points. When $\ep{}$ is small, $P_K(\underline{x})-u(\underline{x})$ dominates the error. This is the flat region in the error as a function of $\ep{}$, see Figures~\ref{fig:errors} and~\ref{fig:errors2}. Then as $\ep{}$ starts to grow, there may be an optimal $\ep{}$-range where the additional terms improve on the polynomial error, but eventually, the $\ep{}$-terms dominate the error, and the exponential convergence rate depends mainly on $\ep{}$ and not on the problem, see Figure~\ref{fig:consterr}.

A contribution that we think is novel and of practical interest is the discussion about convergence for scaled shape parameters. We provide arguments for why $\ep{}=Ch^\beta$ should lead to a convergence rate of the form $e^{C_M/h^{\beta+1}}$, and show that this is what we also get numerically for $\beta=-1/2$.

Another practical contribution is that we have shown that given a reasonable error estimate, we can decide on a good choice for the shape parameter based on a small test problem. Then using a converging shape parameter strategy, we can solve the real problem, and also based on a comparison of error estimates and errors against the finest solution, we can get an improved error estimate for the solution of the most resolved problem. 

Even though global collocation methods are not really practical for large scale problems, many of the things we have learned can be transferred also to localized methods, as these are based on 'local global collocation'.

% Data poly degree $K-Q$ error ep^{2floor(Q/2)+2}

%\section*{Acknowledgements}
%We would like to thank Dr.~Kurt Otto who is responsible for the finite 
%difference code used in the comparisons\ldots

\appendix
\section{Proof sketches}
\label{sec:A}
In order to save space and not repeat already published material, we do not 
give the full proof for Theorem~\ref{th:1} here, instead we
give instructions how to carry out the proof using the machinery laid down 
in~\cite[pp.\ 122--127]{LaFo05}. Because the RBF approximant in the PDE case 
has exactly the same form as the usual RBF interpolant, we get the exact same 
expansion~\cite[Eq.~(28)]{LaFo05} of the solution for small $\ep{}$
\begin{equation}
  s(\underline{x},\ep{})=\ep{-2K}(\ep{-2q}P_{-q}(\underline{x})+\cdots+\ep{2K}P_K(\underline{x})+\cdots).
\end{equation}
What differs from the interpolation case is the conditions that the 
polynomials must fulfill. In the PDE case we have that
\begin{equation}
\begin{array}{ll}
P_K & \mbox{satisfies the inhomogeneous PDE and} \\
    & \mbox{boundary conditions at the $N$ node points},\\
P_j,\quad j\neq K & \mbox{satisfy the homogeneous PDE and}\\
    & \mbox{boundary conditions at the $N$ node points.} 
\end{array}
\label{eq:cond}
\end{equation}
The proof of part (i) is completely analogous to the proofs of Theorems~4.1 
and 4.2 in~\cite{LaFo05}. For part (iii), we follow the steps in the proof of
Theorem~4.1. For simplicity, we first assume that the nullspace $n(\underline{x})$ of the
matrix $Q$ defined in~(\ref{eq:Q}) is of degree $K$. 
The steps are identical
until the point were we are considering the conditions for $P_{-q+K}$.  
There are three possibilities
\begin{itemize}
\item If $q=0$, then the polynomial is $P_K$ and must satisfy the PDE. 
      However, since the matrix $Q$ is singular, this
      can only happen in the (unlikely) case that the right hand side 
      $\underline{f}$ is in the range of $Q$.
\item If $q>0$ and $P_{-q+K}$ is identically zero, then the moment 
      vector $\underline{\sigma}_{-q}$ is zero, leading to 
      $\underline{\lambda}_{-q}$, because of the non-singularity of $P$.
      This means that we could have omitted the ${-q}$ term in the expansion
      and we must have $q=0$. This is in conflict with the previous case.
\item Then we must have $q>0$ and $P_{-q+K}$ must contain a nullspace
      component $\alpha n(\underline{x})$. This means that we have at least
      one divergent term in the expansion of the solution.
\end{itemize}
If there is just a single nullspace component of degree $K$, and extending 
$Q$ with an appropriate monomial of degree $K+1$ leads to 
$\mathrm{rank}(Q)=N$, then at the next step looking at $P_{-q+1+K}$ we get 
the two possibilities $\alpha=0$, which has been ruled out, or 
$P_{-q+1+K}=P_K$. Hence, we must have $q=1$ and divergence of order 
$\ep{2}$.  

If the nullspace is of lower degree than $K$, we will also get divergence,
but the negative power of $\ep{}$ could be higher. 

The argument behind part (ii) is that we need to go to the polynomial
$P_{-q+M}$ before we have enough degrees of freedom to satisfy the discrete
PDE problem. Therefore, the limit must have degree $M$. However, because
$Q$ is non-singular, all previous polynomials must be identically zero and
accordingly there can be no divergence. Compare with the proofs of 
Theorems~4.2 and~4.3. 

For part (iv) of the proof, we follow the proof of Theorem~4.3. The important
difference is that the relation between the moments is determined by the 
nullspace of $P$, but the possible nullspace parts in the polynomials 
$P_{-q+J}$ is determined by the nullspace of $Q$. In~\cite{LaFo05}, we arrive
at an equation $C^TB^{-1}C\underline{\alpha}=\underline{0}$. The corresponding
equation here becomes 
\begin{equation}
C^TB^{-1}D\underline{\alpha}=\underline{0},
\label{eq:CBD}
\end{equation}
where $C$ is of size $n\times m$ and $D$ has dimensions $n\times p$. To be 
precise, at step $J$ of the proof, $m$ is the dimension of the $J$-degree 
part of the nullspace of $P$ and $p$ is the corresponding dimension for $Q$.

If $m=p$, the system~(\ref{eq:CBD}) is square, but non-singularity cannot be
guaranteed when $C$ and $D$ are different. If $m>p$, the system is 
over-determined and it is likely that
the only solution is $\underline{\alpha}=0$. If on the other hand, $m<p$ the
system is under-determined, allowing for non-zero nullspace components in
the expansion polynomials.

If $n(x)$ defines a nullspace component for $P$, then 
$p(\underline{x})n(\underline{x})$ defines a higher degree nullspace component
using any polynomial $p(\underline{x})$. Therefore, the dimension $m$ 
typically grows with $J$. However, there is no similar mechanism for the
nullspace of $Q$ (since $Ln(\underline{x})=0$ does not generally imply  
$L(p(\underline{x})n(\underline{x}))=0$). Accordingly, the dimension
$p$ is likely to stay the same or decrease with $J$.

These facts taken together lead to the statements in part (iv). We use the
formulation {\em likely}, since it should be theoretically possible to construct 
counter examples in both the convergent and the divergent case.  

%\[\|f-s_{f,X_h,K_\epsilon}\|_{L_2} \leq C (h*\epsilon)^{\beta/2} \|f\|_{native space for K}\]

\bibliographystyle{siam}
\bibliography{HZpaper}

\end{document}